\documentclass[12pt]{article}

\usepackage{times}
\usepackage[centertags]{amsmath}
\usepackage{amssymb}
\usepackage{amsthm}
\usepackage{enumitem}
\usepackage[utf8]{inputenc}
\usepackage[T1]{fontenc}
\usepackage[colorlinks=true,urlcolor=red,citecolor=red,linkcolor=red]{hyperref}
\usepackage[a4paper, top=5cm, bottom=5cm, left=3.5cm, right=2.5cm]{geometry}
\usepackage[italian,english]{babel}
\usepackage{float}
\usepackage{mathtools}
\usepackage{graphicx}
\usepackage{caption}
\usepackage{subcaption}
\usepackage{tikz}
\usepackage{tikz-3dplot}
\usepackage{qtree}
\usepackage{tikz-qtree}
\usepackage[lined,algonl,boxed]{algorithm2e}
\usepackage{pgfplots}
\pgfplotsset{compat=1.13}
\usetikzlibrary{trees,positioning,shapes}
\usetikzlibrary{patterns}
\usepackage{verbatim}
\newtheorem{teo}{Theorem}[section]
\newtheorem{defi}[teo]{Definition}
\newtheorem{prop}[teo]{Proposition}
\newtheorem{lem}[teo]{Lemma}
\newtheorem{cor}[teo]{Corollary}
\theoremstyle{definition}
\newtheorem{oss}[teo]{Remark}

\theoremstyle{definition}
\newtheorem{ex}[teo]{Example}

\newcommand{\bs}{\boldsymbol}

\newcommand{\al}{\boldsymbol{\alpha}}
\newcommand{\be}{\boldsymbol{\beta}}
\newcommand{\de}{\boldsymbol{\delta}}
\newcommand{\e}{\boldsymbol{e}}
\newcommand{\om}{\boldsymbol{\omega}}
\newcommand{\g}{\boldsymbol{\gamma}}
\newcommand{\te}{\boldsymbol{\theta}}
\newcommand{\he}{\boldsymbol{\eta}}
\newcommand{\ga}{\boldsymbol{\gamma}}
\newcommand{\eps}{\boldsymbol{\epsilon}}

\newcommand{\Ap}{\mathrm{\bf Ap}}
\newcommand{\wt}{\hspace{0.1cm}\widetilde{\wedge}\hspace{0.1cm}}
\newcommand{\wh}[1]{\widehat{#1}}
\newcommand{\bei}[1]{\bs{\beta^{(#1)}}}
\newcommand{\td}[2]{\widetilde\Delta_{#1}^S(#2)}
\newcommand{\ds}[2]{\Delta_{#1}^S(#2)}
\newcommand{\dE}[2]{\Delta_{#1}^E(#2)}
\newcommand{\tdE}[2]{\widetilde\Delta_{#1}^E(#2)}
\newcommand{\dN}[2]{\Delta_{#1}(#2)}
\newcommand{\tdN}[2]{\widetilde\Delta_{#1}(#2)}

\newcommand{\N}{\mathbb{N}}

{\null\vfill\begin{center}%
		\bfseries Acknowledgements\end{center}}%
{\vfill\null}
\newcommand{\keywords}[1]{\emph{Keywords:} #1}
\newcommand{\MSC}[1]{\emph{Mathematics Subject Classification 2010:} #1}

\newcounter{lastnote}

\title{Partition of complement of good ideals and Ap\'{e}ry sets} 

\author{L. Guerrieri \thanks{Jagiellonian University, Instytut Matematyki, 30-348 Krak\'{o}w \emph{e-mail: guelor@guelan.com}}, N. Maugeri \thanks{Università degli studi di Catania, Dipartimento di Matematica e Informatica, Catania \emph{e-mail: nicola.maugeri.1992@gmail.com}}, V. Micale \thanks{Università degli studi di Catania, Dipartimento di Matematica e Informatica, Catania \emph{e-mail: vmicale@dmi.unict.it}
}
}

\date{}

\begin{document}
\linespread{1,0}
\maketitle
\begin{abstract}
\noindent Good semigroups form a class of submonoids of $\N^d$ containing the value semigroups of curve singularities. In this article, we describe a partition of the complements of good semigroup ideals, having as main application the description of the Ap\'{e}ry sets of good semigroups.
This generalizes to any $d \geq 2$ the results of \cite{DAGuMi}, which are proved in the case $d=2$ and only for the standard Ap\'{e}ry set with respect to the smallest nonzero element. Several new results describing good semigroups in $\N^d$ are also provided.
\end{abstract}
\\
\MSC{20M10, 20M14, 20M25.}

\section*{Introduction}

In \cite{anal:unr} the notion of \emph{good semigroup} was formally defined in order to study value semigroups of Noetherian analytically unramified one-dimensional semilocal reduced rings, e.g. the local rings arising from curve singularities (and from their blowups), possibly with more than one branch. By the truth, properties of these semigroups were already considered in 
\cite{symm:cdk}, \cite{canonical:danna}, \cite{multibranch:delgado}, \cite{symm:delgado},  
but their structure was systematically studied in \cite{anal:unr}. In the one branch case, value semigroups are numerical semigroups, and their theory has been widely studied.

Also in the case with more branches, in which value semigroups are submonoids of $\N^d$, the properties of the correspondent rings can be translated and studied at
semigroup level. For example, the well-known result by Kunz (see \cite{kunz}) that a one-dimensional analytically irreducible local domain is Gorenstein if and only 
if its value semigroup is symmetric can be generalized to analytically unramified rings (see \cite{symm:cdk} and also \cite{symm:delgado}). 

However, good semigroups present some problems that make difficult their study; first of all they are not finitely generated as monoid (even if they
can be completely determined by a finite set of elements (see \cite{garcia1982semigroups}, \cite{Carvalho:Semiring} and \cite{good:danna})) and they are not closed under finite intersections. 
Secondly, their good ideals (i.e. ideals of good semigroups satisfying the same axiomatic definition) do not behave well under sums and differences in the sense that a sum or a difference of two of them does not necessarily produce a good ideal (see e.g. \cite{anal:unr} and \cite{KST}).  


Moreover, the class of good semigroups is larger than the class of value semigroups (see \cite{anal:unr}) and no characterization of value good semigroups is known, unlike the numerical semigroup case in which any such semigroup is the value 
semigroup of the ring of the corresponding monomial curve.
Unfortunately, a consequence of this fact is that, in order to prove some property for a good semigroup, it is not in general possible to take advantage of the 
nature of value semigroups of rings. For this reason, good semigroups with two branches have also been recently studied as a natural generalization of numerical semigroups, using only semigroup techniques, without necessarily referring to the ring context (see \cite{DAGuMi},\cite{emb-NG},\cite{type:good},\cite{}).
More precisely, working in the case of good subsemigroups of $\mathbb N^2$, in \cite{DAGuMi},
the authors study the notion of \emph{Ap\'ery set}, which is in this case infinite, but they show that it has a natural partition in a finite number of sets, called levels, and the number of these sets is
$e=e_1+e_2$
where $(e_1,e_2)$ is the minimal nonzero element in $S$. In case $S$ is a value semigroup of a ring, this number is
equal to the multiplicity of the corresponding ring agreeing with the analogous well-known result in the numerical case. 
Furthermore, for a symmetric good semigroup in $\mathbb N^2$, this partition satisfies a duality property similar to the duality that holds for the Ap\'ery set in the numerical case. 
The partition in levels of the Ap\'ery set seems to be very useful also to study other properties of good semigroups. For example, in \cite{emb-NG} the authors define the concept of embedding dimension for a good semigroup and use the levels of the Ap\'ery set to prove that the embedding dimension is bounded above by the multiplicity (as it happens for numerical semigroups and for one-dimensional rings). 

Finally, in \cite{type:good} the authors define the same partition in levels for complements of proper good ideals (hence also for Ap\'ery sets with respect to any element), in order to define and study the type of a good semigroup and to study almost symmetric good semigroups and the symmetry conditions on their Ap\'ery sets.
Following this approach, our purpose is to generalize to semigroups contained in $\N^d$ some of the notions and properties defined in the case $d=2$. In particular the main aim of this paper is to generalize the partition in levels for complements of good ideals and to show that, 
in the case of a principal good ideal $E$ generated by an element $\om$ (i.e. the Ap\'ery set with respect to $\om$ is the complement of $E$), the number of levels is exactly the sum of the components of $\om$ (see, Theorem \ref{main}). The proof of Theorem \ref{main}  gives also an alternative proof for \cite[Theorem 3]{DAGuMi} in the case $d=2$. We do not use an extension of the same argument of \cite[Theorem 3]{DAGuMi}, since for $d > 3$ that would require too many technicalities and, instead, we develop a new method which provides also more information about the structure of a good semigroup and gives a consistent generalization of \cite[Theorem 5]{DAGuMi}.

The structure of this paper is the following. In Section 1, we recall the preliminary definitions and we prove several general results needed for the rest of the paper. In particular, we introduce the concept of \it complete infimum \rm and we study the relation that this has with Property (G2) of good ideals (see Definition \ref{completeinfimum} and Proposition \ref{propG2}). This relation will be essential in the next section. 

In Section 2, we introduce the previously mentioned partition in levels of complement of good ideals according to what has been done in \cite{DAGuMi} and \cite{type:good} in case $d=2$ (a similar partition was firstly defined in \cite{Apery:danna} for the Ap\'{e}ry set of the value semigroup of a plane curve). This will be defined using a partial order relation and the notion of complete infimum. Finally, in order to compute the number of levels, we generalize \cite[Lemma 2.3]{type:good} to $\mathbb N^d$ (see, Theorems \ref{bianchi} and \ref{neri}).

In Section 3, we define the subspaces of a good semigroup, whose name arises from the fact that they geometrically represent discrete subspaces contained in the good semigroup. For a value semigroup of an analytically irreducible ring, this definition has a precise correspondence in ring theory which is pointed up in \cite[Corollary 1.6]{canonical:danna}. Identifying a subspaces with a particular representative element, we generalize to the set of subspaces many properties proved in the two preceding sections. 
 
Finally, in Section 4, we prove Theorem \ref{main} computing the number of levels of a the complement of a good ideal in the case of principal good ideals. 

\section{Good Semigroups with $d$ branches}

Let $\mathbb N$ be the set of nonnegative integers. As usual, $\le$ stands for the natural partial ordering in $\mathbb N^d$: set $\al=(\alpha_1, \alpha_2,\ldots,\alpha_d), \be=(\beta_1, \beta_2,\ldots,\beta_d)$, then $\boldsymbol{\alpha}\le \boldsymbol{\beta}$ if $\alpha_i\le \beta_i$ for all $i\in \{1,\ldots,d\}$. Trough this paper, if not differently specified, when referring to minimal or maximal elements of a subset of $\mathbb N^d$, we refer to minimal or maximal elements with respect to $\le$.
Given $\boldsymbol{\alpha},\boldsymbol{\beta}\in \mathbb N^2$, the infimum of the set $\{\boldsymbol{\alpha},\boldsymbol{\beta}\}$ (with respect to $\le$) will be denoted by
\[\boldsymbol{\alpha}\wedge \boldsymbol{\beta}=(\min(\alpha_1,\beta_1),\min(\alpha_2,\beta_2),\ldots,\min(\alpha_d,\beta_d)).\]

Let $S$ be a submonoid of $(\mathbb N^d,+)$. We say that $S$ is a \emph{good semigroup} if

\begin{itemize}
	\item[(G1)] For every $\boldsymbol{\alpha},\boldsymbol{\beta}\in S$, $\boldsymbol{\alpha}\wedge \boldsymbol{\beta}\in S$;
	\item[(G2)] Given two elements $\boldsymbol{\alpha},\boldsymbol{\beta}\in S$ such that $\al\neq \be$ and $\alpha_i=\beta_i$ for some $i\in\{1,\ldots,d\}$, then there exists $\bs{\epsilon}\in S$ such that $\epsilon_i>\alpha_i=\beta_i$ and $\epsilon_j\geq \min\{\alpha_j,\beta_j\}$ for each $j\neq i$ (and if $\alpha_j\neq \beta_j$ the equality holds).
	\item[(G3)] There exists an element $\boldsymbol{c}\in S$ such that $\boldsymbol{c}+\mathbb N^d\subseteq S$.
\end{itemize}

Sometimes, inspired by the notion of value semigroups of algebroid curves with $d$ branches, we will say that $S$ has $d$ branches if it a submonoid of $\mathbb N^d$. \\
A good subsemigroup is said to be \emph{local} if $\boldsymbol{0}=(0,\ldots,0)$ is its
only element with a zero component. In the following, if not otherwise specified, all the results will be proved for any good semigroup, independently whether local or not. 

Notice that, from condition (G1), if $\boldsymbol{c}$ and $\boldsymbol{d}$ fulfill (G3), then so does $\boldsymbol{c}\wedge \boldsymbol{d}$. So there exists a minimum $\boldsymbol{c}\in \mathbb N^2$ for which condition (G3) holds. Therefore we will say that
\[
\boldsymbol{c}:=\min\{\boldsymbol{\alpha}\in \mathbb Z^d\mid \boldsymbol{\alpha}+\mathbb N^d\subseteq S\}
\]
is the \emph{conductor} of $S$.
We denote $\boldsymbol{\gamma}:=\boldsymbol{c}-\textbf{1}$.

A subset $E \subseteq \N^d$ is a \it relative ideal \rm of $S$ if $E+S \subseteq E$ and there exists $\al \in S$ such that $\al + E \subseteq S$. A relative ideal $E$ contained in $S$ is simply called an ideal. A (relative) ideal $E$
satisfying properties (G1) and (G2) is said to be a (relative) \it good ideal \rm of S (any good relative ideal satisfies automatically property (G3)). In this article we will usually work with proper ideals $E \subseteq S$. By properties (G1) and (G3), there exist a minimal element $\boldsymbol{c}_E$ such that $\boldsymbol{c}_E + \N^d \subseteq E$. Such element is called the \it conductor \rm of $E$. As for $S$, we denote $\boldsymbol{\gamma}_E:=\boldsymbol{c}_E-\textbf{1}$.

By property (G1), if $S$ is local, $S$ has a unique 
  minimal nonzero element that we denote by $\bs{e}=(e_1,e_2,\ldots,e_d)$. \\
The set $\e + S$ is a good ideal of $S$ and its conductor is $\boldsymbol{c} + \e$. Similarly for every $\om \in S$, the principal good ideal $E= \om + S$ has conductor $\boldsymbol{c}_E = \boldsymbol{c} + \om$.

\medskip


Through this paper we set $I=\{1,\ldots, d\}$ where $d$ is the number of branches of the considered semigroups. The following notation is introduced for any subset $S \subseteq \N^d$ (but this will be usually applied to good semigroups or good ideals).
Let $F\subseteq I$, $\al\in \N^d$, then we set:
\begin{eqnarray*}
    \Delta^S_F(\al)&=&\{\be\in S \hspace{0.1cm}|\hspace{0.1cm} \beta_i=\alpha_i, \text{ for any } i\in F \text{ and } \beta_j>\alpha_j \text{ for } j\notin F\}.\\
    \widetilde{\Delta}^S_F(\al)&=&\{\be\in S \hspace{0.1cm}|\hspace{0.1cm} \beta_i=\alpha_i, \text{ for any } i\in F \text{ and } \beta_j\geq\alpha_j \text{ for } j\notin F\}\setminus\{\al\}.\\
    \Delta^S_i(\al)&=&\{\be\in S\hspace{0.1cm}|\hspace{0.1cm} \beta_i=\alpha_i, \mbox{ and } \beta_j>\alpha_j \mbox{ for any } j\neq i\}\\
    \Delta^S(\al)&=&\bigcup_{i=1}^d\Delta_i^S(\al)
\end{eqnarray*}
In particular, for $S=\N^d$, we set:
   \[ \Delta_F(\al):=\Delta_F^{\N^d}(\al)=\{\be\in \N^d \hspace{0.1cm}|\hspace{0.1cm} \beta_i=\alpha_i, \text{ for any } i\in F \text{ and } \beta_j>\alpha_j \text{ if } j\notin F\}.\]

Given $F\subseteq I$, we denote by $\wh{F}$ the set $I \setminus F$. We call $\wh{F}$ the \emph{orthogonal} set of $F$.

This notation has a precise geometrical meaning which is showed, for $d=3$, in Figure \ref{fig:notations}. In particular:\\
In (a), $\dN{F}{\al}$ is represented as an half-line; in this case $\dN{\wh F}{\al}$ is the orthogonal open half-plane. By definition this one does not include the two half-line which delimit the half plane.\\
In (b), $\dN{F}{\al}$ is again as an half-line and $\tdN{\wh F}{\al}$ is the orthogonal half-plane including also the two half-line which delimit it.\\
In (c), $\Delta(\al)$ is represented as union of the three open half-planes $\dN{1}{\al}$,$\dN{2}{\al}$,$\dN{3}{\al}$.

\begin{figure}[H]
\centering
\caption{\footnotesize{A graphical representation of the notation introduced}}
\begin{subfigure}{.35\textwidth}
\tdplotsetmaincoords{60}{120}
\begin{tikzpicture}[tdplot_main_coords, scale=0.8, font=\footnotesize]
\draw[thick,dashed,->] (0,0,0) -- (3.5,0,0) node[anchor= east]{$1$};
\draw[thick,dashed,->] (0,0,0) -- (0,2.8,0) node[anchor= west]{$2$};
\draw[thick,->] (0,0,0) -- (0,0,3) node[anchor=south]{$3$};

\draw plot [mark=*, mark size=1] coordinates{(0,0,0)}  node at (0,0.25,0.25) {$\bs{\alpha}$} node[anchor= south west] at (0,0,1.8) {$\dN{F}{\al}$} node at (1.75,1.4,0) {$\dN{\wh F}{\al}$};  

\filldraw[draw=white, fill=black!20, fill opacity=0.3]         
            (0,0,0)
            -- (3.5,0,0)
            -- (3.5,2.8,0)
            -- (0,2.8,0)
            -- cycle; 
   
\end{tikzpicture}
 \caption{\footnotesize{Representation of $\dN{F}{\al}$}}
\end{subfigure}%
\begin{subfigure}{.35\textwidth}
\tdplotsetmaincoords{60}{120}
\begin{tikzpicture}[tdplot_main_coords, scale=0.8, font=\footnotesize]

\filldraw[draw=white, fill=black!20, fill opacity=0.3]         
            (0,0,0)
            -- (3.5,0,0)
            -- (3.5,2.8,0)
            -- (0,2.8,0)
            -- cycle; 
            
\draw[thick,->] (0,0,0) -- (3.5,0,0) node[anchor= east]{$1$};
\draw[thick,->] (0,0,0) -- (0,2.8,0) node[anchor= west]{$2$};
\draw[thick,->] (0,0,0) -- (0,0,3) node[anchor=south]{$3$};

\draw plot [mark=*, mark size=1] coordinates{(0,0,0)}  node at (0,0.25,0.25) {$\bs{\alpha}$} node[anchor= south west] at (0,0,1.8) {$\dN{F}{\al}$} node at (1.75,1.4,0) {$\tdN{\wh F}{\al}$};  
\end{tikzpicture}
 \caption{\footnotesize{Representation of $\tdN{F}{\al}$}}
\end{subfigure}%
\begin{subfigure}{.35\textwidth}
 \tdplotsetmaincoords{60}{120}
\begin{tikzpicture}[tdplot_main_coords, scale=0.8, font=\footnotesize]

\filldraw[draw=white, fill=black!20, fill opacity=0.3]         
            (0,0,0)
            -- (3.5,0,0)
            -- (3.5,2.8,0)
            -- (0,2.8,0)
            -- cycle; 
            
\filldraw[draw=white, fill=black!20, fill opacity=0.3]         
            (0,0,0)
            -- (0,2.8,0)
            -- (0,2.8,3)
            -- (0,0,3)
            -- cycle; 

\filldraw[draw=white, fill=black!20, fill opacity=0.3]         
            (0,0,0)
            -- (0,0,3)
            -- (3.5,0,3)
            -- (3.5,0,0)
            -- cycle; 
            
\draw[thick,->] (0,0,0) -- (3.5,0,0) node[anchor= east]{$1$};
\draw[thick,->] (0,0,0) -- (0,2.8,0) node[anchor= west]{$2$};
\draw[thick,->] (0,0,0) -- (0,0,3) node[anchor=south]{$3$};

\draw plot [mark=*, mark size=1] coordinates{(0,0,0)}  node at (0,0.25,0.25) {$\bs{\alpha}$}  node at (1.75,1.4,0) {$\dN{3}{\al}$} node at (1.75,0,1.5) {$\dN{2}{\al}$} node at (0,1.4,1.5) {$\dN{1}{\al}$}; 
\end{tikzpicture}
 \caption{\footnotesize{Representation of $\Delta(\al)$}}
\end{subfigure}%
\label{fig:notations}
\end{figure}
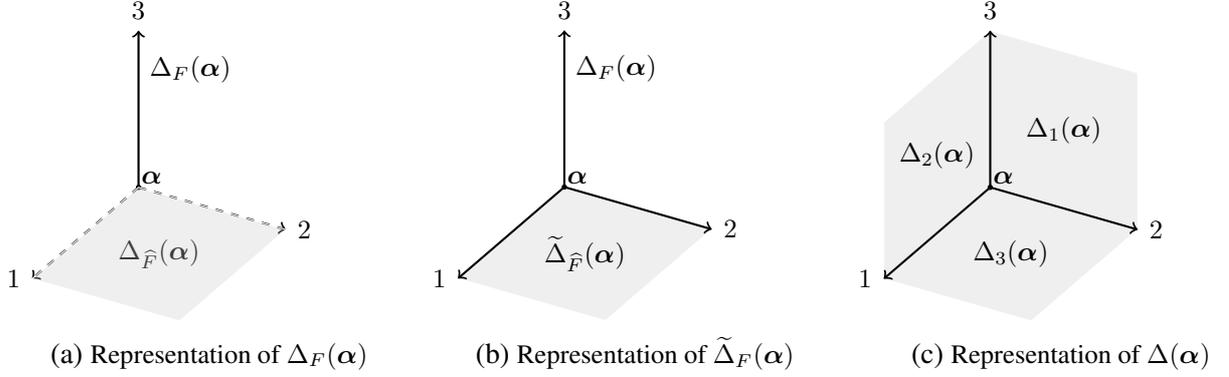

\begin{oss}
\label{remdelta}
Let $F,G \subseteq I$ and $E \subseteq \N^d$. We list some observations following directly from the preceding definitions:
\begin{enumerate}
    \item \label{remdelta1} $\be\in \widetilde{\Delta}^E_F(\al)$ if and only if there exists $G\supseteq F$, such that $\be\in \Delta^E_G(\al)$. 
    \item \label{remdelta2} $G\supseteq F$ if and only if, for every $\al \in \mathbb{N}^d$, $\tdE{G}{\al}\subseteq \tdE{F}{\al}$.
    \item \label{remdelta3} If $\te \in \Delta^E_{F}(\al)$, then $\tdE{G}{\te} \subseteq \dE{G}{\al}$ for every $G \supseteq F$.
    \item \label{remdelta4} Assume that $E$ is a good ideal. As a consequence of property (G2), if $\al\in E$ and $\dE{F}{\al}\neq \emptyset$, then $\tdE{\wh{F}}{\al}\neq \emptyset$.
    Equivalently, there exists $G\supseteq \wh{F}, G\neq I$, such that $\dE{G}{\al}\neq \emptyset$.
\end{enumerate}
\end{oss}

\medskip

We introduce the following terminology in order to have a better control on the use of  property (G2) and simplify our notation.

\begin{defi} \rm \label{completeinfimum}
Given a good semigroup $S\subseteq \N^d$, and $A\subseteq S$ be any subset. We say that $\al\in A$ is a \it complete infimum \rm in $A$ if there exist $\be^{(1)},\ldots,\be^{(r)}\in A$, with $r \geq 2$, satisfying the following properties:
\begin{enumerate}
\item $\be^{(j)} \in \ds{F_j}{\al}$ for some non-empty set $F_j \subsetneq I$.
\item For every $j\ne k \in \lbrace 1, \ldots, r \rbrace$, we have $\al = \be^{(j)} \wedge \be^{(k)} $.
\item $\bigcap_{k=1}^r {F_k}= \emptyset$.
\end{enumerate}
In this case we write $\al=\be^{(1)}\wt\be^{(2)}\cdots\wt\be^{(r)}$.
\end{defi}
When $d=2$, the infimum of every two incomparable elements (with respect to "$\leq$") is a complete infimum of them.
More in general, given $\al \in S$, if there exists $\be \in \ds{F}{\al}$ and $\de\in \ds{\wh{F}}{\al}$, then $\al=\be\wt\de$ is the complete infimum of $\be$ and $\de$ in $S$.
For another example assume for some $\al \in S$, that $\ds{H}{\al} \neq \emptyset$ for every set $H$ of cardinality $d-1$. Then $\al$ is a complete infimum of a set of elements, each one picked from one of these sets of maximal cardinality. In this case $\al$ is a complete infimum of exactly $d$ elements.

\medskip

In this section we prove that Property (G2) has a useful application involving the notion of complete infimum (see, Proposition \ref{propG2}). Before doing that, we will need to prove several facts. We start by the next easy observation:

\begin{oss} \label{remarkG2}
Given a good ideal $E \subseteq S$, $\al \in E$, $F\subsetneq I$ and $\be \in \dE{F}{\al}$, there exist $\be^{(1)},\ldots,\be^{(r)}$ with $1 \leq r \leq |F|$, such that:  
	\begin{itemize}
	\item 
	For $i=1, \ldots r$, $\be^{(i)} \in \dE{G_i}{\al}$, with $G_i \supseteq \wh{F}$.
	\item 
	$ G_1 \cap G_2 \cap \cdots \cap G_r = \wh{F}. $
	\end{itemize}
	
Indeed, for every $i \in F$, $\alpha_i=\beta_i$, and applying property (G2) to $\al$ and $\be$ with respect to $i$, we can find $\be^{(i)} \in \dE{G_i}{\al}$ with $G_i \supseteq \wh F$ and $i \not \in G_i$. By relabeling the indexes we may assume $i\in \lbrace 1,\ldots,r \rbrace $ with $r \leq |F|$. 

\end{oss}

Notice that the elements $\be^{(i)}$ in the above remark may not be all distinct and therefore $\al$ is not necessarily a complete infimum of them and $\be$. We are going to prove that we can find an opportune collection of $\be^{(i)}$ fulfilling also this extra condition in Proposition \ref{propG2}. 

As consequence of the preceding remark we have the following useful fact:

\begin{prop} \label{usodiG2}
Let $\al \in E$, $F \subsetneq I$ a set of indexes and assume there exists a set $H$ of cardinality $d-1$ containing $F$ and such that $\dE{G}{\al} = \emptyset$ for every $F \subseteq G \subseteq H$. Then $\dE{\widehat{F}}{\al} = \emptyset$.
\end{prop}
\begin{proof}
Set $\wh H = \lbrace i \rbrace$.
 Assume by way of contradiction there exists $\be \in \dE{\widehat{F}}{\al}$. Using Remark \ref{remarkG2}, since $i \in \wh F$, we can find $\te \in \dE{G}{\al}$ with $G \supseteq F$ and $i \not \in G$. It follows $G \subseteq H$ and this is a contradiction.
\end{proof}

\medskip

Next lemma contains more basic results that will be widely used during this article. For any subset $A \subseteq S$, we say that two elements $\al, \be \in A$ are \it consecutive \rm in $A$ if whenever $\al \leq \de \leq \be$ for some $\de \in A$, then $\de = \al$ of $\de = \be$.

\begin{lem} \label{minimidelta} Let $E \subseteq S$ be a proper good ideal and let $\al \in E$. Then:
\begin{enumerate}
    \item \label{minimidelta1}If $\be \in \dE{F}{\al}$ and $\te \in \dE{G}{\al}$, then $\be \wedge \te \in \dE{F \cup G}{\al}$ if $F \cup G \subsetneq I$ and $\be \wedge \te = \al$ if $F \cup G = I$.
    \item \label{minimidelta2}Let $\be \in \dE{F}{\al}$ be consecutive to $\al$ in $E$. Then $\dE{H}{\al}= \emptyset$ for every $H \supsetneq F$.
    \item \label{minimidelta3}Assume $\dE{F}{\al} = \emptyset$ and $F=G_1 \cup G_2$. Then either $\dE{G_1}{\al} = \emptyset$ or $\dE{G_2}{\al} = \emptyset$.
    \item \label{minimidelta4}Assume there exists $F \subsetneq I$ such that $\dE{H}{\al} = \emptyset$ for every $H \supsetneq F$. Then $\dE{H}{\al} = \emptyset$ for every $H \subsetneq \wh F$.
    \item \label{minimidelta5}Assume there exists a maximal set $F \subsetneq I$ such that $\dE{F}{\al} \neq \emptyset$ ($F$ is not necessarily unique). 
    If
    $\dE{H}{\al} \neq \emptyset$ then either $H \subseteq F$ or $H \supseteq \wh F$. 

\end{enumerate}
\end{lem}

\begin{proof}
1. Set $\he := \be \wedge \te $. The thesis follows since, by definition, $\eta_i=\alpha_i$ if and only if $i \in F \cup G$. \\
2. Suppose $\de \in \dE{H}{\al}$ for some $H \supsetneq F$. The element $\te:=\be \wedge \de \in \dE{H}{\al}$ by item 1. Hence $\al < \te < \be$ and this is a contradiction since $\al$ and $\be$ are consecutive in $E$. \\
3. It is a straightforward consequence of item \ref{minimidelta1}. \\
4. It is an easy consequence of Proposition \ref{usodiG2}. \\
5. By assumption $\dE{H}{\al} = \emptyset$ for every $H \supsetneq F$ and, by item \ref{minimidelta4}, also for every $H \subsetneq \wh F$. Assume $H$ to be not comparable by inclusion with both $F$ and $\wh F$. It follows that $F \subsetneq (F \cup H) \subsetneq I$. By item \ref{minimidelta3}, necessarily $\dE{H}{\al} = \emptyset$.
\end{proof}

As application we discuss a key consequence of property (G1) of good ideals.

\begin{oss} \label{rem1} 
Let $E \subseteq S$ be a proper good ideal and let $\al \in S \setminus E$.
    Assume $\dE{F}{\al} \neq \emptyset$. As consequence of property (G1) and of Lemma \ref{minimidelta}.\ref{minimidelta1}, $\tdE{\widehat{F}}{\al} = \emptyset. $ 
\end{oss}

We are now ready to describe the relation between property (G2) and complete infimums.
For semplicity we will denote by $$R_i:=I\setminus \{i\}=\{1,2,\ldots,i-1,i+1,\ldots,d\}, \hspace{2cm} i=1,\ldots, d.$$ 

\begin{prop}
\label{propG2}
Let $S \subseteq \N^d$ be a good semigroup, $E \subseteq S$ a good ideal and take $\al \in E$.  Suppose there exists $\be \in \dE{F}{\al}$ for some $F\subsetneq I$. Then, there exist $\be^{(1)},\cdots,\be^{(r)}$ with $1 \leq r \leq |F|$, such that  $$\al = \be \wt \be^{(1)}\wt\be^{(2)}\wt\cdots\wt\be^{(r)}.$$
	In particular $\be^{(i)} \in \dE{G_i}{\al}$, with $G_i \supseteq \wh{F}$ and
	$ G_1 \cap G_2 \cap \cdots \cap G_r = \wh{F}. $
\end{prop}

\begin{proof}
Applying property (G2) as in Remark \ref{remarkG2}, for every $i \in F$, we can find 
a maximal set $G_i$ such that $\widehat{F} \subseteq G_i \subseteq R_i = I \setminus \{ i \}$ and $\dE{G_i}{\al} \neq \emptyset$. For a fixed index $i \in F$, such set $G_i$ is unique since, if there were two sets $G_i$,$G_i'$ fulfilling the same condition, then also their union $G_i \cup G_i' \subseteq R_i$ and, by Lemma \ref{minimidelta}.1, $\dE{G_i \cup G_i'}{\al} \neq \emptyset$ contradicting the assumption of maximality.
Obviously, by relabeling the indexes, we can always assume $F= \{ 1, \ldots, r\}$, call those sets $G_1, \ldots, G_r$ and by construction we have $ G_1 \cap G_2 \cap \cdots \cap G_r = \wh{F}. $ Observe that the last condition on the intersection is satisfied also if we have $G_i=G_j$ for some $i,j \in F$ and in this case we consider this set only once. To conclude that $\al$ is a complete infimum of elements chosen with respect to such sets we still need to prove condition 2 of Definition \ref{completeinfimum}, that, again by Lemma \ref{minimidelta}.1, is equivalent to show that $G_i \cup G_j = I$ for every choice of $i,j \in F$ such that $G_i \neq G_j$. \\
Assume this to be not true by way of contradiction. Hence there exists an index $k \in F$ such that $k \not \in G_i \cup G_j$ and moreover $\dE{G_i \cup G_j}{\al} \neq \emptyset$. Thus $G_i \subsetneq G_i \cup G_j \subseteq G_k$ and therefore we can restrict to show that we cannot have a proper containment $G_j \subsetneq G_i$ for any $i,j \in F$. If this was the case, by maximality of $G_j$, this would imply $j \in G_i$. Consider now a set $T$ such that $\widehat{G_i} \subseteq T \subseteq R_j$. Clearly $T \nsubseteq G_j$ otherwise we would have the contradiction $\widehat{G_i} \subseteq T \subseteq G_j \subsetneq G_i$, and therefore $G_j \subsetneq G_j \cup T \subseteq R_j$. By maximality of $G_j$ this implies $\dE{G_j \cup T}{\al} = \emptyset$, and by Lemma \ref{minimidelta}.3, this implies also $\dE{T}{\al} = \emptyset.$ Now, Proposition \ref{usodiG2} implies that $\dE{G_i}{\al} = \emptyset$ and this is a contradiction. 
\end{proof}

\medskip 

Now we want to generalize to any good semigroup a property already known in the case $d=2$. As in \cite{good:danna} we denote by $\operatorname{Small}(S)=\{\al\in S:$ $\al\leq \bs{c}\}$ the set of the \emph{Small Elements} of $S$. As a consequence of the following proposition, the Small Elements determine all the elements of the semigroup.

\begin{prop}
\label{infsem}
Let $S\subseteq \N^d$ be a good semigroup, $E\subseteq S$ a good ideal and pick $\al\in \N^d$. Suppose $\bs{c}_E=(c_1,\ldots,c_d)\in \dE{F}{\al}$, for some non-empty set $F\subsetneq I$. The following assertions are equivalent:
\begin{enumerate}
     \item \label{infsem1} $\al\in E$. 
    \item \label{infsem2} $\widetilde{\Delta}_{\wh{F}}(\al) \subseteq E$.
    \item \label{infsem3}  $\tdE{\wh{F}}{\al}\neq \emptyset$.
\end{enumerate}
\end{prop}
\begin{proof}
\ref{infsem1} $\Rightarrow$ \ref{infsem2}) Let $\be\in \widetilde{\Delta}_{\wh{F}}{(\al)}$. Since $\bs{c}_E \in \dE{F}{\al}$ and $\al \in E$, by 
 Property (G2) there exists $\de\in \tdE{\wh{F}}{\al}$. We consider three cases.\\
\bf Case 1: $\de\geq \be.$ \rm In this case $\delta_i=\beta_i=\alpha_i$ for all $i\in \wh{F}$ and $\delta_i\geq \beta_i\geq \alpha_i$ for all $i\in F$. Now we find $\eps\in \N^d$ such that $\epsilon_i=c_i > \alpha_i$ for all $i\in \wh{F}$ and $\epsilon_i=\beta_i\geq \alpha_i=c_i$ for all $i\in F$. Since $e_i\geq c_i$ for every $i\in I$, by definition of conductor, $\eps\in E$; hence $\be=\eps\wedge \de\in E$.\\
\bf Case 2: $\de<\be.$ \rm We set $\de^{(1)}:=\de$ and we consider $\eps^{(1)}$ such that $\epsilon^{(1)}_i=c_i$ for all $i\in \wh{F}$ and $\epsilon^{(1)}_i=\delta^{(1)}_i\geq \alpha_i=c_i$ for all $i\in F$. As before, since $\eps^{(1)}\geq \bs{c}_E$, $\eps^{(1)}\in E$. Furthermore $\eps^{(1)}\in \ds{F}{\de^{(1)}}$, hence, by Remark \ref{remdelta}.\ref{remdelta4}, $\tdE{\wh{F}}{\de^{(1)}}\neq \emptyset$. Therefore there exists $\de^{(2)}>\de^{(1)}$, such that $\de^{(2)}\in \tdE{\wh{F}}{\de^{(1)}}$. If $\de^{(2)}\geq \be$ we conclude as in Case 1; otherwise we can repeat the previous construction on $\de^{(2)}$ finding an element $\de^{(3)}>\de^{(2)}$. 
After a finite number of iterations, if we never conclude by Case 1, we find a maximal element $\de^{(k)} \in E$ such that  $\de^{(k)} \leq \be$. Considering the corresponding element $\eps^{(k)}\in \ds{F}{\de^{(k)}}$ and applying Property (G2) following Proposition \ref{propG2}, we write  
	$$\de^{(k)} = \eps^{(k)} \wt \he^{(1)}\wt\he^{(2)}\wt\cdots\wt\he^{(r)}$$ with $\he^{(i)} \in \dE{G_i}{\de^{(k)}}$, and
$ G_1 \cap G_2 \cap \cdots \cap G_r = \wh{F}. $ 
If  $\he^{(i)} \geq \be$ for some $i$, we conclude using Case 1. In the opposite case we will get a contradiction. Indeed, for every $i$, we have $\de^{(k)} \leq \he^{(i)} \wedge \be < \be$. Say that $\be \in \Delta_H(\de^{(k)})$ for $H \supseteq \wh F$ and notice that there must exists $i$ such that $G_i \cup H \subsetneq I$, otherwise we would have the contradiction $$\wh H \subseteq G_1 \cap G_2 \cap \cdots \cap G_r = \wh{F} \subseteq H. $$ Hence, set $G:= G_i$ and $\he:= \he_i$. For $j \not \in G \cup H$, we have $\eta_j, \beta_j > \delta^{(k)}_j$ and hence $\de^{(k)} < \he \wedge \be < \be$. Since $\he \in E$, by Case 1, $\he \wedge \be \in E$ and this contradicts the assumptions of maximality of $\de^{(k)}$. \\
\bf Case 3: $\de \wedge \be < \be, \de$. \rm In this case, applying Case 1 to $\de \wedge \be$ and $\de$, we get $\de \wedge \be \in E$. Then we use Case 2. \\
\ref{infsem2} $\Rightarrow$ \ref{infsem3}) It is straightforward. \\
\ref{infsem3} $\Rightarrow$ \ref{infsem1}) If $\be\in \tdE{\wh{F}}{\al}\neq \emptyset$, we have that $\beta_i=\alpha_i$ for all $i\in \wh{F}$ and $\beta_i\geq \alpha_i=c_i$ for all $i\in F$.
Since $\bs{c}_E\in \dE{F}{\al}$, we have $\bs{c}_E \wedge \be=\al\in E$. 
\end{proof}

The preceding proposition can be rephrased in the following way.
For $\al\in \N^d$ such that $\alpha_i=c_i$ for every $i$ in a non-empty set $F$ we have that
$\al \in E$ if and only if all the elements $\be \in \N^d$ such that $\beta_i \geq c_i$ for $i\in F$ and $\beta_i=\alpha_i$ for $i\in \wh{F}$ are also in $E.$

\medskip

\begin{ex}
\label{ex3rami}
We represented graphically in Figure \ref{fig:semex}, the semigroup $S\subseteq \N^3$ having:
\begin{align*}
    \operatorname{Small}(S)=&&\{(1,2,3),(1,2,6),(1,2,7),(1,2,8),(2,3,3),(2,3,6),(2,3,7),(2,4,3),\\
    &&(2,4,6),(2,4,9)(3,3,3),(3,3,6),(3,3,7),(3,5,3),(3,5,6),(3,5,9)\}
\end{align*}
It has conductor $\bs{c}=(3,5,9)$ and $\bs{\gamma}=(2,4,8)$.


\begin{figure}[H]
\centering
 \tdplotsetmaincoords{70}{120}
\begin{tikzpicture}[tdplot_main_coords, scale=0.6, font=\small]

\draw[thick,->] (0,0,0) -- (15,0,0) node[anchor= east]{$1$};
\draw[thick,->] (0,0,0) -- (0,13,0) node[anchor= west]{$2$};
\draw[thick,->] (0,0,0) -- (0,0,16) node[anchor=south]{$3$};

\filldraw[draw=black, fill=black!20, fill opacity=0.5]         
            (3,5,3)
            -- (9,5,3)
            -- (9,11,3)
            -- (3,11,3)
            -- cycle; 
\filldraw[draw=black, fill=black!20, fill opacity=0.3]         
            (3,5,6)
            -- (3,11,6)
            -- (9,11,6)
            -- (9,5,6)
            -- cycle; 
\filldraw[draw=black, fill=black!20, fill opacity=0.3]         
            (3,5,9)
            -- (3,11,9)
            -- (3,11,15)
            -- (3,5,15)
            -- cycle; 
\filldraw[draw=black, fill=black!20, fill opacity=0.3]         
            (3,5,9)
            -- (9,5,9)
            -- (9,5,15)
            -- (3,5,15)
            -- cycle; 
\filldraw[draw=black, fill=black!20, fill opacity=0.3]         
            (3,5,9)
            -- (9,5,9)
            -- (9,11,9)
            -- (3,11,9)
            -- cycle; 
            
\draw[-] (2,4,9) -- (2,4,15); 
\draw[-] (3,3,3) -- (9,3,3); 
\draw[-] (3,3,6) -- (9,3,6); 
\draw[-] (3,3,7) -- (9,3,7); 

\draw plot [mark=*, mark size=1] coordinates{(0,0,0)};
\draw plot [mark=*, mark size=1] coordinates{(1,2,3)};
\draw plot [mark=*, mark size=1] coordinates{(1,2,6)};
\draw plot [mark=*, mark size=1] coordinates{(1,2,7)};
\draw plot [mark=*, mark size=1] coordinates{(1,2,8)};
\draw plot [mark=*, mark size=1] coordinates{(2,3,3)};
\draw plot [mark=*, mark size=1] coordinates{(2,3,6)};
\draw plot [mark=*, mark size=1] coordinates{(2,3,7)};
\draw plot [mark=*, mark size=1] coordinates{(2,4,3)};
\draw plot [mark=*, mark size=1] coordinates{(2,4,6)};
\draw plot [mark=*, mark size=1] coordinates{(2,4,9)};
\draw plot [mark=*, mark size=1] coordinates{(3,3,3)};
\draw plot [mark=*, mark size=1] coordinates{(3,3,6)};
\draw plot [mark=*, mark size=1] coordinates{(3,3,7)};
\draw plot [mark=*, mark size=1] coordinates{(3,5,3)};
\draw plot [mark=*, mark size=1] coordinates{(3,5,6)};
\draw plot [mark=*, mark size=1] coordinates{(3,5,9)} node[anchor= south west ]{$\bs{c}$};

\draw (1,-0.1,0) -- ++ (0,0.2,0) node[anchor= east] at (0,-0.5,0) {$1$};
\draw (2,-0.1,0) -- ++ (0,0.2,0) node[anchor= east] at (1,-0.5,0) {$2$};
\draw (3,-0.1,0) -- ++ (0,0.2,0) node[anchor= east] at (2,-0.5,0) {$3$};

\draw (-0.2,2,0) -- ++ (0.3,0,0) node[anchor= west] at (-1,1.3,0) {$2$};
\draw (-0.2,3,0) -- ++ (0.3,0,0) node[anchor= west] at (-1,2.3,0) {$3$};
\draw (-0.2,4,0) -- ++ (0.3,0,0) node[anchor= west] at (-1,3.3,0) {$4$};
\draw (-0.2,5,0) -- ++ (0.3,0,0) node[anchor= west] at (-1,4.3,0) {$5$};

\draw (-0.1,0,3) -- ++ (0.2,0,0) node[anchor= west] at (0,0,3) {$3$};
\draw (-0.1,0,6) -- ++ (0.2,0,0) node[anchor= west] at (0,0,6) {$6$};
\draw (-0.1,0,7) -- ++ (0.2,0,0) node[anchor= west] at (0,0,7) {$7$};
\draw (-0.1,0,9) -- ++ (0.2,0,0) node[anchor= west] at (0,0,9) {$9$};

\end{tikzpicture}
  \caption{\footnotesize{The semigroup $S$ of the Example \ref{ex3rami}}}
  \label{fig:semex}
\end{figure}
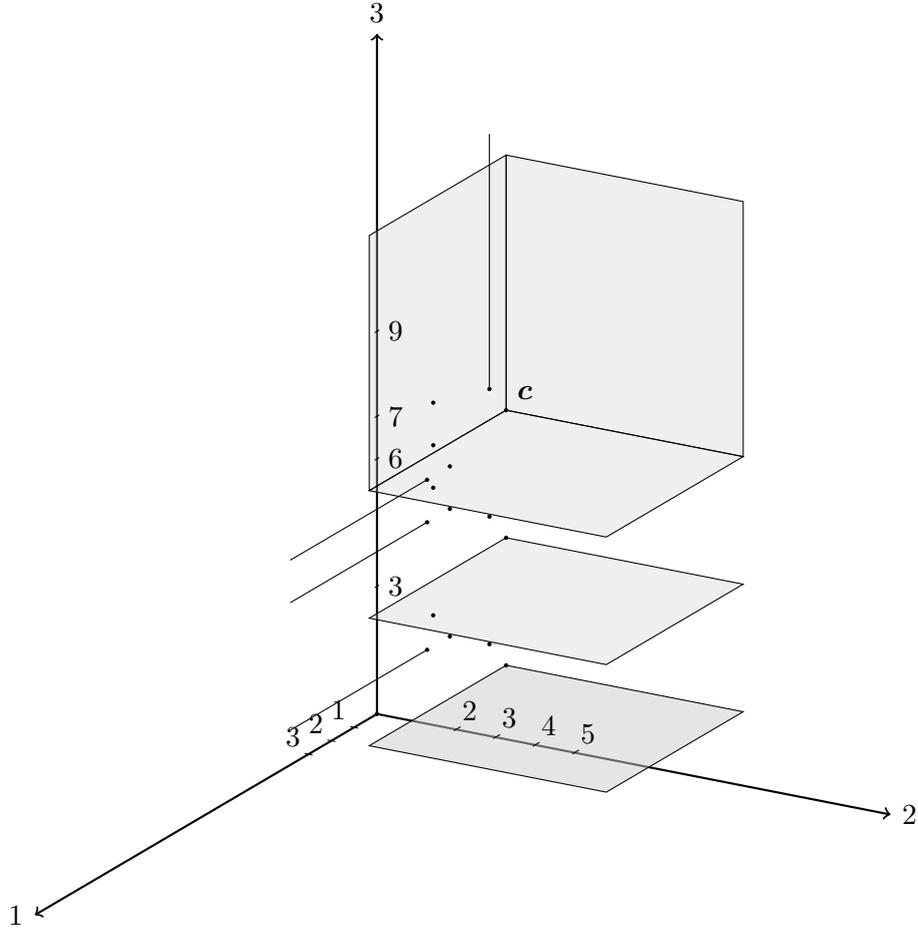

\end{ex}

\section{Partition of the complement of a good ideal}

Let $S \subseteq \N^d$ be a good semigroup and 
let $E \subseteq S$ be a proper good ideal. We call $A:=S \setminus E$ its complementary set. In this section, we introduce a partition of $A$ in a finite number of subsets, according to what has been done in \cite{DAGuMi} and \cite{type:good} in case $d=2$ (a similar partition was firstly defined in \cite{Apery:danna} for the Ap\'{e}ry set of the value semigroup of a plane curve).
In analogy to the definition given in these works, we will call \emph{levels} the elements of the partition. The partition will be defined using the definition of \emph{complete infimum} given in the previous section. 

Moreover, we need to use another order relation on $\mathbb N^d$ (already introduced in \cite{anal:unr} and in the successive works). Given $\al=(\alpha_1,\alpha_2,\ldots, \alpha_d)$ and $\be= (\beta_1,\beta_2,\ldots,\beta_d)$ in $\mathbb{N}^d$, we say that  $\al \leq \leq \be$ if and only if either $\al = \be$ or $\alpha_i<\beta_i$ for every $i\in\{1,\ldots, d\}$. When the second case is verified we say that $\be$ \it dominates \rm $\al$ and use the notation $\al \ll \be$.


We define the partition in levels of $A \subsetneq S$. In general, $A$ will be the complementary set of a good ideal, but the only assumption really needed to define this partition is the existence of a conductor for the set $S \setminus A$ (property (G3)). 

\begin{defi} \rm \label{deflivelli}
Set: 
$$ B^{(1)}:=\{\al \in A : \al \  \mbox{is maximal with respect to} \le\le\},$$
$$C^{(1)}:= \{ \al \in B^{(1)} : \al=\be^{(1)}\wt\be^{(2)}\cdots\wt\be^{(r)} \mbox{ where } 1<r\leq d,\mbox{ and } \bei{k} \in B^{(1)}\},$$
$$ D^{(1)}:=B^{(1)}\setminus C^{(1)}.$$
Assume $i>1$ and that $D^{(1)},\dots , D^{(i-1)}$ have been defined. Set inductively:
$$ B^{(i)}:=\{\al \in A \setminus (\bigcup_{j < i} D^{(j)}) : \al \  \mbox{is maximal with respect to} \le\le\},$$
$$C^{(i)}:= \{ \al \in B^{(i)} : \al=\be^{(1)}\wt\be^{(2)}\cdots\wt\be^{(r)} \mbox{ where } 1<r\leq d,\mbox{ and } \bei{k} \in B^{(i)}\},$$
$$D^{(i)}:=B^{(i)}\setminus C^{(i)}.$$ 
By construction $D^{(i)}\cap D^{(j)}=\emptyset$, for any $i\neq j$ and, 
since the set $S \setminus A$ has a conductor, there exists $N \in \mathbb N_+$ such that $A=\bigcup_{i=1}^N D^{(i)}$.
For simplicity, we prefer to number the set of the partition in increasing order with respect to the components of their elements,
thus we set $A_i:= D^{(N+1-i)}$. Hence 
$$A=\bigcup_{i=1}^N A_i. $$
We call the sets $A_i$ the \it levels \rm of $A$.
\end{defi}

Given $\om \in S$, we can consider the good ideal $E=\om+S$. In this case its complement $A= S \setminus E=\Ap(S, \om)$ is the Ap\'ery set of $S$ with respect to $\om$. The main aim of this article is to prove that the number of levels of an Ap\'ery set is equal to the sum of the components of the element $\om$. 

\begin{oss} \label{lastlev}
As observed in \cite[Lemma 1]{DAGuMi} in the case $d=2$, it is straightforward to see that also for $d > 2$, if $\al, \be \in A$, $\al \ll \be$ and $\al \in A_i$, then $\be \in A_j$ for some $j>i$. Moreover, the last set of the partition is $A_N=\Delta(\ga_E)=\Delta^S(\ga_E)$
while, if $S$ is local, $A_1=\{\boldsymbol {0}\}$. 
\end{oss}

The following results generalize several properties of the levels of $A$ proved in \cite[Lemma 1]{DAGuMi} in the case $d=2$ and $A= \Ap(S, \om)$.

\begin{lem}
\label{prelem}
The sets $A_i$ satisfy the following properties: 
\begin{enumerate}
    \item \label{prelem1} Given $\al\in A_i$ with $i < n$, either there exists $\be\in A_{i+1}$ such that $\al\ll\be$ or,  
  $$\al=\be^{(1)}\wt\be^{(2)}\cdots\wt\be^{(r)}$$ where $1<r\leq d$, $\bei{k}\in A$, and at least one of them belong to $A_{i+1}$.
    \item \label{prelem2} Given $\al\in A_i$, with $i\neq N$, there exists $\be\in A_{i+1}$ such that $\be\geq \al$.
    \item \label{prelem3} For every $\al\in A_i$ and $\be\in A_j$, with $j\geq i$, $\be \not \ll \al$.
    \item \label{prelem4} Given $\al\in A_i$ and $\be\in A$ such that $\be \geq \al$, then $\be\in A_{i}\cup\cdots\cup A_N$.
    \item \label{prelem5}
    Let $\al= \be^{(1)}\wt\be^{(2)}\cdots\wt\be^{(r)} \in A$ and assume that for every $k$, $\bei{k} \in A_i$, then $\al \in A_h$ with $h < i$.
    \item \label{prelem6} Assume $\al, \be \in A$ are consecutive in $S$ or in $A$. If $\al\ll \be$, then there exists $i$ such that $\al \in A_i$ and $\be \in A_{i+1}$. If $\be\in \ds{F}{\al}$ for some $F \subsetneq I$, then there exists $i$ such that either $\al\in A_i$ and $\be\in A_{i+1}$ or $\al,\be\in A_{i}$.
 \end{enumerate}
\end{lem}

\begin{proof}
1,2,3,4 and 5 follow directly from Definition \ref{deflivelli}. The proof of 6 is analogous to case of good semigroups of $\N^2$ (see \cite[Lemma 1]{DAGuMi}). 
\end{proof}

\medskip

The main goal of this section is to generalize to good semigroups in $\N^d$ the following proposition, very helpful to have some control on the levels of different elements. 

\begin{prop}\label{altro-Ap}\rm (\cite[Lemma 2.3]{type:good}) \it 
Let $S \subseteq \N^2$ be a good semigroup. Let $A \subseteq S$ be such that $E:=S \setminus A$ is a proper good ideal of $S$ and let $A=\bigcup_{i=1}^N A_{i}$ be its partition in levels. The following assertions hold:
\begin{enumerate}
    \item Assume $\al \in S$, $\be \in \ds{j}{\al} \cap A_i$ and $\ds{3-j}{\al} \subseteq A$ (recall that in this case $j \in \{ 1,2 \}$). Let $\te \in \ds{3-j}{\al}$ be consecutive to $\al$ (such element exists by property (G2)). Then $\te \in A_h$ with $h \leq i.$
    \item Assume $\al \in A_i$, $\be \in \dE{j}{\al}$ and let $\te \in \ds{3-j}{\al}$ be consecutive to $\al$. Then $\te \in A_i$.
\end{enumerate}
\end{prop}

The two results of this Proposition \ref{altro-Ap} will be generalized in Theorems \ref{bianchi} and \ref{neri}. We prove the first one (see Figure \ref{fig:white} for a graphical representation in case $d=3$).

\begin{teo}

\label{bianchi}
Let $S$ be a good semigroup, $E \subseteq S$ a good ideal and $A= S \setminus E$.
 Let $\al\in S$, $\be\in \Delta_F^S(\al)\cap A_i$ and assume $\ds{\widehat{F}}{\al}\subseteq A$. Let $\bs{\theta}\in \ds{G}{\al}$ with $\bs{\theta}$ and $\al$ consecutive and $G \supseteq \widehat{F}$.
\begin{enumerate}
\item If $\td{G}{\al} \subseteq A$, then $\bs{\theta}\in A_h$ with $h\leq i$;
\item If $\al\in A$ and $\td{\wh F}{\al}\subseteq A$ then $\al\in A_h$ with $h<i$.
\end{enumerate}

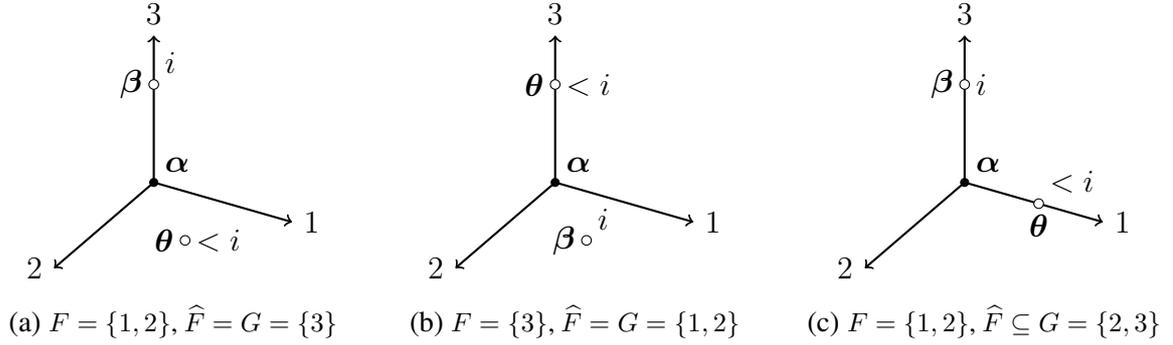
\begin{figure}[H]
\begin{subfigure}{.33\textwidth}
  \centering
 \tdplotsetmaincoords{60}{120}
\begin{tikzpicture}[scale=0.75,tdplot_main_coords]
\draw[thick,->] (0,0,0) -- (3.5,0,0) node[anchor= east]{$2$};
\draw[thick,->] (0,0,0) -- (0,2.8,0) node[anchor= west]{$1$};
\draw[thick,->] (0,0,0) -- (0,0,3) node[anchor=south]{$3$};
\draw [fill=white] (0,0,2) circle[radius= 0.2 em] node[anchor= east]{$\bs{\beta}$} node[anchor= south west]{$i$}; 
\draw plot [mark=*, mark size=2] coordinates{(0,0,0)}  node[anchor= south west]{$\bs{\alpha}$};  
\draw [fill=white] (1.5,1.5,0) circle[radius= 0.2 em]  node[anchor= east]{$\bs{\theta}$} node[anchor= west]{$<i$};     
\end{tikzpicture}
  \caption{\footnotesize{$F=\{1,2\}$, $\wh F=G=\{3\}$}}
  \label{fig:whitesub1}
\end{subfigure}%
\begin{subfigure}{.33\textwidth}
  \centering
 \tdplotsetmaincoords{60}{120}
\begin{tikzpicture}[scale=0.75,tdplot_main_coords]
\draw[thick,->] (0,0,0) -- (3.5,0,0) node[anchor= east]{$2$};
\draw[thick,->] (0,0,0) -- (0,2.8,0) node[anchor= west]{$1$};
\draw[thick,->] (0,0,0) -- (0,0,3) node[anchor=south]{$3$};

\draw [fill=white] (0,0,2) circle[radius= 0.2 em] node[anchor= east]{$\bs{\theta}$} node[anchor=  west]{$<i$}; 
\draw plot [mark=*, mark size=2] coordinates{(0,0,0)}  node[anchor= south west]{$\bs{\alpha}$};  
\draw [fill=white] (1.5,1.5,0) circle[radius= 0.2 em]  node[anchor= east]{$\bs{\beta}$} node[anchor= south west]{$i$};  
\end{tikzpicture}
\caption{\footnotesize{$F=\{3\}$, $\wh F=G=\{1,2\}$}}
\label{fig:whitesub2}
\end{subfigure}
\begin{subfigure}{.33\textwidth}
  \centering
 \tdplotsetmaincoords{60}{120}
\begin{tikzpicture}[scale=0.75,tdplot_main_coords]
\draw[thick,->] (0,0,0) -- (3.5,0,0) node[anchor= east]{$2$};
\draw[thick,->] (0,0,0) -- (0,2.8,0) node[anchor= west]{$1$};
\draw[thick,->] (0,0,0) -- (0,0,3) node[anchor=south]{$3$};

\draw [fill=white] (0,0,2) circle[radius= 0.2 em] node[anchor= east]{$\bs{\beta}$} node[anchor=  west]{$i$}; 
\draw plot [mark=*, mark size=2] coordinates{(0,0,0)}  node[anchor= south west]{$\bs{\alpha}$};  
\draw [fill=white] (0,1.5,0) circle[radius= 0.2 em]  node[anchor= north]{$\bs{\theta}$} node[anchor= south west]{$<i$};  
\end{tikzpicture}
\caption{\footnotesize{$F=\{1,2\}$, $\wh F\subseteq G=\{2,3\}$}}
\label{fig:whitesub3}
\end{subfigure}
\caption{\footnotesize{Two representation of the setting of the theorem in case $d=3$}}
\label{fig:white}
\end{figure}

\end{teo}
\begin{proof}
1. We work by reverse induction on $i$. If $i=N$ there is nothing to prove. We assume the thesis true for $t>i$ and we prove it for $i$.
By Lemma \ref{prelem}. \ref{prelem2}, there exists $\de\in A_{i+1}$ such that $\de\geq \be\geq \al$, furthermore $\te\geq \al$, hence $\al\leq \de\wedge\te\leq \te$. The hypothesis of $\al$ and $\te$ consecutive implies that either $\de\wedge\te=\al$ or $\de\wedge\te=\te$.

First we suppose that for every $\de\in A_{i+1}$, with $\de\geq \be$ we have $\de\wedge\te=\al$. Take one of such elements $\de$ and observe that necessarily $\de \in \ds{T}{\be}$ for some $T \supseteq \wh G$, otherwise $\de\wedge\te$ would be equal to $\te$. Moreover, in this case there are no elements $\de$ in the level $A_{i+1}$ such that $\de \gg \be$ and by Lemma \ref{prelem}.\ref{prelem1}, $\be$ is a complete infimum of some elements in $A$ and we can include $\de$ among these elements.  
Hence there exists $\om\in A\cap \ds{H}{\be}$ for some $H \supseteq \wh T$ and we may also choose $\om$ consecutive to $\be$ in $A$ and such that $\wh G \nsubseteq H$.  
It follows that there exists $j \in \wh G \setminus H$ and thus $\te_j > \al_j$ and $\om_j > \be_j \geq \al_j$. This implies $\te \wedge \om = \te$ and hence $\om \in A_{i}$ since we assumed that no element satisfying the same property was in $A_{i+1}$. The thesis follows since now $\te \leq \om$.
\\
As second case, we can suppose that there exists $\de\in A_{i+1}$ with $\de\geq \be$, such that $\de\wedge \te=\te$; in particular $\de\geq \te$.
If $\de \gg \te$, by definition of levels (see Lemma \ref{prelem}.\ref{prelem3}), $\te\in A_h$ with $h\leq i$ and this concludes the proof. 

Otherwise there exists $H\subseteq F$ such that $\de\in \ds{H}{\te}$. 
Now observe that if $\eps \in \td{\wh F}{\te}$ and $\eps_i= \al_i$, then $i \in G$.
Hence, as consequence of Remark \ref{remdelta}, $$\td{\widehat{H}}{\te} \subseteq \td{\wh F}{\te} \subseteq \ds{\widehat{F}}{\al} \cup \td{G}{\te} \subseteq A, $$ 
and,
applying 
Proposition \ref{propG2} 
to $\te$ and $\de$ (choosing the good ideal in that proposition to be $S$ itself), we obtain that $\te$ can be expressed as complete infimum of $\de$ and other elements $\om^1, \ldots, \om^r \in \td{\widehat{H}}{\te} \subseteq A$. 
Without loss of generality we may assume that for every $j=1,\ldots, r$, $\om^j$ and $\te$ are consecutive and, applying the inductive hypothesis on $\de, \te, \om^j$, we get $\om^j \in A_h$ with $h \leq i+1$. If for some $j$ we have $h \leq i$, we are done since $\te \leq \om^j$, otherwise, we must have $\om^1, \ldots, \om^r \in A_{i+1}$ and therefore $\te$ is a complete infimum in $A$ of elements of $A_{i+1}$. Hence, by Lemma \ref{prelem}.\ref{prelem5}, $\te$ must be contained in a level $A_h$ for $h \leq i$. \\ 
2. By Proposition \ref{propG2} applied to $\al$ and $\be$ (again with respect to the good ideal $S$), the element $\al$ can be expressed as complete infimum of $\be$ and other elements in $\te^1, \ldots, \te^r \in \td{\widehat{F}}{\al} \subseteq A$. For every $j=1, \ldots, r$ we assume without loss of generality that $\al$ and $\te^j$ are consecutive. Applying the first part of this lemma, we get $\te^j \in A_h$ with $h \leq i$. As observed previously, if for some $j$ we have $h < i$, we are done since $\al \leq \te$, otherwise $\te^1, \ldots, \te^r \in A_{i}$ and $\al$ is a complete infimum in $A$ of elements of $A_{i}$. This implies the thesis 
again using Lemma \ref{prelem}.\ref{prelem5}. 
\end{proof}

Before proving the second important result of this section we need to discuss a couple of extra properties. The first one shows how the use of Property (G2) in order to find elements in $\dE{}{\al}$ works when passing from $\al$ to some $\be \in \dE{F}{\al}$.

\begin{prop}
\label{riferimento}
Let $E \subseteq S$ be a good ideal and let $\al \in S$. Assume there exists $\be \in \dE{F}{\al}$ and that $\dE{H}{\al}$ is non-empty for some $H \subsetneq F$. Then there exists $T \subsetneq F$ such that $T \supseteq (F \setminus H)$ and $\dE{T}{\al} \neq \emptyset.$
\end{prop}
\begin{proof}
Let $\de \in \dE{H}{\al}$ and assume $\be$ to be minimal in $\dE{F}{\al}$. Since $H \subsetneq F$, and by minimality of $\be$ together with Lemma \ref{minimidelta}.1, we get $\be \leq \de$ and therefore $\de \in \dE{U}{\be}$ for some $U \supseteq H$ such that $U \cap F = H$ 
(observe that for $j \in H$, $\de_j= \al_j= \be_j$ while for $j \in F \setminus H$, $\de_j > \al_j=\be_j$). Applying Property (G2) to $\de$ and $\be$ as in Proposition \ref{propG2}, we can find $\he \in \dE{V}{\be}$ for some $V \supseteq \wh U$ such that $V \nsupseteq F$. Clearly $\he \geq \be \geq \al$ and $\he_j= \be_j = \al_j$ if $j \in F \cap V $ (which is non empty since $F \nsubseteq U$).
Hence $\he \in \dE{T}{\al}$ such that $$T \supseteq F \cap V \supseteq F \cap \wh U \supseteq (F \setminus H). $$
Furthermore, for $j \in (F \setminus V)$, $\he_j > \be_j = \al_j$ and for $j \not \in F$, $\he_j \geq \be_j > \al_j$. This implies $T \subsetneq F$.
\end{proof}

The proof of Theorem \ref{neri} is strongly based on next lemma.

\begin{lem}
\label{preblack}
Let $S$ be a good semigroup, $E \subseteq S$ a good ideal and $A= S \setminus E$.
Let $\al\in A_i$ and assume $ \Delta_F^E(\al) \neq \emptyset$. 
Let $\te \in \td{\widehat{F}}{\al}$ with $\te$ and $\al$ consecutive. Then $\te \in A_i$ if one of the following conditions is satisfied:
\begin{enumerate}
    \item \label{preblack1} There exists $\de \in A_{i+1}$ such that $\de \gg \te$. 
    \item \label{preblack2} There exists $\de \in A_{i+1}$ such that $\de \in \ds{H}{\te}$ with $H \subseteq F$. 
    \item \label{preblack3} There exists $\de \in A_i \cap \Delta_F^S(\al)$. 
    \item \label{preblack4} There exists $\de \in A_{i+1} \cap \Delta_F^S(\al)$ such that $\de \leq \be \in \tdE{F}{\al}$. 
    \item \label{preblack5} There exists $\de \in A_{i+1} \cap \Delta_F^S(\al)$ such that $\tdE{F}{\de}= \emptyset$.
\end{enumerate}
\end{lem}
\begin{proof}
First of all, we observe that under these assumptions, by Remark \ref{rem1}, $\td{\widehat{F}}{\al} \subseteq A$. Furthermore, say that $\te \in \ds{G}{\al}$ for some $G \supseteq \wh F$. We discuss each case separately. \\ 
1. In this case the thesis follows immediately by Lemma \ref{prelem}.\ref{prelem3}.\\
2. Given $\eps \in \td{\widehat{H}}{\te}$, notice that for $i \in \wh H$, $\eps_i= \te_i$. 
Hence for $i \in \wh F \subseteq \wh H$, $\eps_i= \te_i= \al_i$.
It follows that $\td{\widehat{H}}{\te} \subseteq \td{\widehat{F}}{\al} \subseteq A$
and we can apply the second statement of Theorem \ref{bianchi} to $\te$ and $\de$ to argue that $\te \in A_i$ (it cannot stay in a lower level since $\te \geq \al$).\\
3. We get the thesis applying Theorem \ref{bianchi} to $\de, \al, \te$. \\
4. We work by reverse induction on $i$. If $i=N$ the thesis is clear, hence assume it to be true for any $t>i$ and prove it for $i$. 
In this case we want to find an element $\om \in A_{i+1}$ such that $\om \geq \te$ and if $\om_i = \te_i$ then $i \in F$. Hence such element $\om$ will satisfy the assumptions of \ref{preblack1} or \ref{preblack2}.
Let $\be \in \dE{U}{\de}$ with $U \supseteq F$ and apply property (G2) to $\be$ and $\de$ following Proposition \ref{propG2}. Thus there exist some elements $\om^1, \ldots, \om^k \in S$ with $1 \leq k < d$ such that $\om^i \in \ds{V_i}{\de} $ and $V_1 \cap V_2 \cap \cdots \cap V_k = \wh U$. Since $\wh U \subseteq \wh F \subseteq G$, there exists $i$ such that $\wh V_i \cap \wh G \neq \emptyset$. Set $V:=V_i$ and $\om=\om^i$. By Remark \ref{rem1}, $\td{\widehat{U}}{\de} \subseteq A$ and hence, since $V \supseteq \wh U$, $\om \in A $ and we may assume without restrictions $\om$ and $\de$ to be consecutive. Applying the inductive hypothesis to $\de$ and $\om$, we get $\om \in A_{i+1}$. \\
Now we show $\om \geq \te$. Clearly $\om \geq \de \geq \al$ and thus $\om \wedge \te \geq \al$. But for our choice of $V$, we can find $j \in \wh V \cap \wh G$ and therefore $\om_j > \de_j \geq \al_j$ and $\te_j > \al_j$. It follows that $\om \wedge \te > \al$ and hence, since $\al$ and $\te$ are consecutive, $\om \wedge \te = \te$  and $ \te \leq \om.$ Moreover, if $i \in \wh F \subseteq G$, $\om_i \geq \de_i > \al_i= \te_i$ and hence if $\om_i = \te_i$, then $i \in F$. \\
5. In this case, eventually changing our choice of $\de$, we can find $\be \in \dE{F}{\al}$ such that $\be \leq \de$ and they are consecutive (if there is some element in $A_i$ between them, we conclude using \ref{preblack3}). 

Hence, there exists $U \supseteq F$ such that $\de \in \ds{U}{\be}$ and furthermore $\Delta_U(\be)= \Delta_U(\de) \cup \lbrace \de \rbrace$ implying $\dE{U}{\be}= \emptyset$. Lemma \ref{minimidelta}.2 implies that, if $T \supsetneq U$, then $\dE{T}{\be} \subseteq \ds{T}{\be} = \emptyset$ and, by Proposition \ref{usodiG2}, it follows that $\dE{\wh U}{\be} = \emptyset$.
Using Property (G2) on $\de$ and $\be$ and the same procedure used in \ref{preblack4}, we can find $\om \in \ds{V}{\be}$ with $V \supseteq \wh U$ and $\wh V \cap \wh G\neq \emptyset$. In particular 
$\om \geq \te$ and if $\om_i = \te_i$ then $i \in F$. 
In the case in which $\ds{V}{\be} \subseteq A$, assuming $\om$ and $\be$ to be consecutive and since also $\ds{\wh U}{\be} \subseteq A$, it is possible to apply Theorem \ref{bianchi} to $\de, \be, \om$ to get $\om \in A_j $ with $i \leq j \leq i+1$ and conclude using \ref{preblack2}. 

Hence, to complete this part of the lemma, we can suppose that $\dE{V}{\be} \neq \emptyset$ and that without loss of generality $\om \in E$. 
By Proposition \ref{propG2}, applied now to $\be$ and $\om$, we find that there exist $\bs{\tau^1},\ldots, \bs{\tau^k}\in \dE{W_i}{\be}$ with $W_i\supseteq \wh V$ and $W_1\cap\cdots\cap W_t=\wh V$. Since $\wh U \subseteq V$, there exists an index $j\in\{1,\ldots,k\}$ such that $\wh W_j\cap \wh U\neq \emptyset$. Set $W:=W_j$ and $\bs{\tau}:=\bs{\tau^j}$. 
Observe that $\be \leq \bs{\tau}\wedge \de\leq \de$ and furthermore, since $\wh U \cap \wh W\neq \emptyset$, 
there exists $l \in I \setminus (U\cup W)$ and hence $\bs{\tau}_l, \de_l > \be_l$. It follows
$\bs{\tau}\wedge \de \neq \be$ and, since $\be$ and $\de$ are consecutive, necessarily $\bs{\tau}\wedge \de=\de$ and hence $\bs{\tau} \geq \de \geq \be$. 

Moreover, for every $i\in W$, $\tau_i=\delta_i=\be_i$ and thus 
$\bs{\tau}\in \dE{K}{\de}$ for some $K\supseteq W \supseteq \wh V$. 

Applying again Property (G2) on $\bs{\tau},\de$, following Proposition \ref{propG2} (and choosing $S$ itself as good ideal), we find $\bs{\rho^1},\ldots, \bs{\rho^t}\in \ds{Z_i}{\de}$ with $Z_i\supseteq \wh K$ and $Z_1\cap\cdots\cap Z_t=\wh K$. Since $$\emptyset \neq \wh V \cap \wh G \subseteq  K \cap \wh G \subseteq K,$$ again as before, there exists an index $j\in \{1,\ldots, t\}$ such that $(K \cap \wh G) \cap \wh Z_j \neq \emptyset$ 
and we set $Z:=Z_j$ and $\bs{\rho}:=\bs{\rho^j}$. Since $Z\supseteq \wh K$, we get $\bs{\rho}\in \ds{Z}{\de}\subseteq \td{\wh K}{\de}\subseteq A$, where the last inclusion follows by Remark \ref{rem1} since $\bs{\tau} \in E$. Applying the inductive hypothesis on $\bs{\tau},\de, \bs{\rho}$ we conclude that $\rho\in A_{i+1}$. 
We want to prove that $\bs{\rho}$ is an element satisfying the assumptions of \ref{preblack2}. 
First we show $\te\leq \bs{\rho}$. 
Notice that $\al \leq \de \leq \bs{\rho}$ and hence $\al \leq \bs{\rho}\wedge \te \leq \te$. 
For our choice of $Z$, there exists $l \in \wh G \cap \wh Z \neq \emptyset$ and hence $\te_l > \al_l$ and $\bs{\rho}_l > \de_l \geq \al_l$, implying $\bs{\rho}\wedge \te \neq \al$.
Since $\al$ and $\te$ are consecutive, $\bs{\rho}\wedge \te=\te $ and $ \te \leq \bs{\rho}$. 

Now we prove that if $\rho_i=\te_i$ then $i\in F$. Indeed, if $i\in \wh F \subseteq G $, we would have $\rho_i\geq \delta_i\geq \beta_i>\alpha_i=\theta_i$. This proves our claim and conclude the proof of this lemma. 
\end{proof}

We are finally ready to generalize the second part of Proposition \ref{altro-Ap} (see Figure \ref{fig:black} for a graphical representation of the theorem in case $d=3$).
\begin{teo}
\label{neri} 
Let $S$ be a good semigroup, $E \subseteq S$ a good ideal and $A= S \setminus E$.
Let $\al\in A_i$ and let $\te \in \ds{G}{\al}$ consecutive to $\al$. Assume that $\tdE{\wh G}{\al} \neq \emptyset$. Then $\te \in A_i$. 
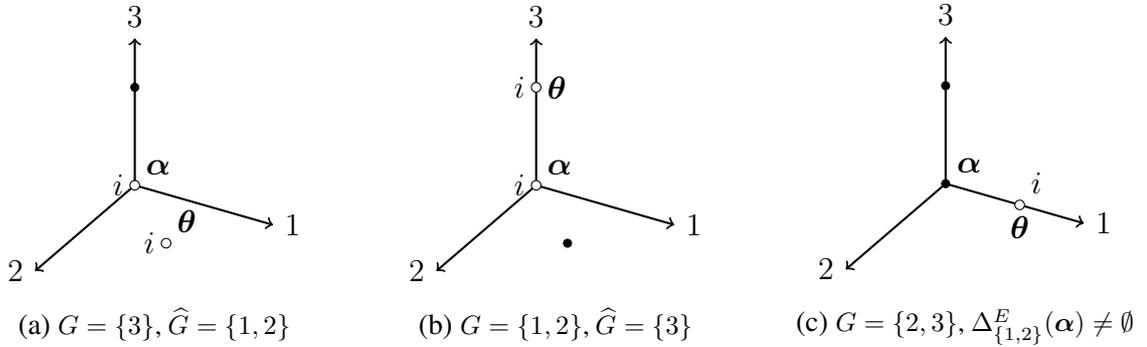
\begin{figure}[H]
\begin{subfigure}{.33\textwidth}
  \centering
 \tdplotsetmaincoords{60}{120}
\begin{tikzpicture}[scale=0.75,tdplot_main_coords]
\draw[thick,->] (0,0,0) -- (3.5,0,0) node[anchor= east]{$2$};
\draw[thick,->] (0,0,0) -- (0,2.8,0) node[anchor= west]{$1$};
\draw[thick,->] (0,0,0) -- (0,0,3) node[anchor=south]{$3$};
\draw plot [mark=*, mark size=2] coordinates{(0,0,2)};
\draw [fill=white] (1.5,1.5,0) circle[radius= 0.2 em]  node[anchor= south west]{$\bs{\theta}$} node[anchor= east]{$i$};  
\draw [fill=white] (0,0,0) circle[radius= 0.2 em]  node[anchor= south west]{$\bs{\alpha}$} node[anchor= east]{$i$};    
\end{tikzpicture}
  \caption{\footnotesize{$G=\{3\}$, $\wh G=\{1,2\}$}}
  \label{fig:blacksub1}
\end{subfigure}%
\begin{subfigure}{.33\textwidth}
  \centering
 \tdplotsetmaincoords{60}{120}
\begin{tikzpicture}[scale=0.75,tdplot_main_coords]
\draw[thick,->] (0,0,0) -- (3.5,0,0) node[anchor= east]{$2$};
\draw[thick,->] (0,0,0) -- (0,2.8,0) node[anchor= west]{$1$};
\draw[thick,->] (0,0,0) -- (0,0,3) node[anchor=south]{$3$};

\draw plot [mark=*, mark size=2] coordinates{(1.5,1.5,0)}; 
\draw [fill=white] (0,0,2) circle[radius= 0.2 em]  node[anchor= west]{$\bs{\theta}$} node[anchor= east]{$i$};  
\draw [fill=white] (0,0,0) circle[radius= 0.2 em]  node[anchor= south west]{$\bs{\alpha}$} node[anchor= east]{$i$};    
\end{tikzpicture}
\caption{\footnotesize{$G=\{1,2\}$, $\wh G=\{3\}$}}
  \label{fig:blacksub2}
\end{subfigure}
\begin{subfigure}{.33\textwidth}
  \centering
 \tdplotsetmaincoords{60}{120}
\begin{tikzpicture}[scale=0.75,tdplot_main_coords]
\draw[thick,->] (0,0,0) -- (3.5,0,0) node[anchor= east]{$2$};
\draw[thick,->] (0,0,0) -- (0,2.8,0) node[anchor= west]{$1$};
\draw[thick,->] (0,0,0) -- (0,0,3) node[anchor=south]{$3$};

\draw plot [mark=*, mark size=2] coordinates{(0,0,2)};  
\draw plot [mark=*, mark size=2] coordinates{(0,0,0)}  node[anchor= south west]{$\bs{\alpha}$};  
\draw [fill=white] (0,1.5,0) circle[radius= 0.2 em]  node[anchor= north]{$\bs{\theta}$} node[anchor= south west] at (0,1.5,0){$i$};  
\end{tikzpicture}
\caption{\footnotesize{$G=\{2,3\}$, $\dE{\{1,2\}}{\al} \neq \emptyset$}}
\label{fig:blacksub3}
\end{subfigure}
\caption{\footnotesize{Two representations of the setting of the theorem in case $d=3$}}
\label{fig:black}
\end{figure}
\end{teo}

\begin{proof}
Since for $i=N$ the proof is straightforward, we may assume the thesis to be true for $t>i$ and prove it for $i$.

Let $F$ be any set such that $\dE{F}{\al} \neq \emptyset$ and $\wh G \subseteq F$. By Remark \ref{rem1}, $\te \in \td{\wh F}{\al} \subseteq A$. 
Suppose there exists a set $H$ such that $F \subsetneq H \subsetneq I$, $\dE{H}{\al} = \emptyset$ but there exists $\he \in \ds{H}{\al}$ consecutive to $\al$. Hence $\he \in A_j$ with $i \leq j \leq i+1$. Consider an element $\be \in \dE{F}{\al}$. Now there exists $U \supseteq F$ such that $\be \in \dE{U}{\he}$ and we follow the same argument based on property (G2) used in the proof of Lemma \ref{preblack}.\ref{preblack4} in order to find $\om \in \td{\wh U}{\he} \subseteq A$ consecutive to $\he$ and such that either $\om \gg \te$ or $\om \in \ds{T}{\te}$ for some $T \subseteq F$. If $\he \in A_{i+1}$ then $\om \in A_h$ with $i \leq h \leq i+1$ and, if $h=i$ we conclude since $\om \geq \te$. Otherwise, 
if $\he \in A_{i+1}$, the inductive hypothesis applied to $\he$ and $\om$ forces $\om \in A_{i+1}$. In this case $\om$ satisfies the assumptions of Lemma \ref{preblack} (\ref{preblack1} or \ref{preblack2}) with respect to $\te$ and we can conclude using them.

From this fact, it follows that we may restrict to choose $F$ to be maximal with respect to having $\dE{F}{\al} \neq \emptyset$ and $\ds{H}{\al}= \emptyset$ for every $F \subsetneq H \subsetneq I$. Lemma \ref{minimidelta}.5, applied choosing as good ideal $S$ itself, implies that if, for some set $H$, $\ds{H}{\al} \neq \emptyset$ then either $H \subseteq F$ or $H \supseteq \wh F$. Analogously, since $\al$ and $\te$ are consecutive, as consequence of Lemma \ref{minimidelta}.2, we also know that if $\ds{H}{\al} \neq \emptyset$ then either $H \subseteq G$ or $H \supseteq \wh G$. 
    
Also notice that if there exists $\de \in A_{i+1}$ such that $\de \gg \al$, then, since $\al$ and $\te$ are consecutive, $\de \geq \te$ and if $\de_i=\te_i$ then $i \in \wh G \subseteq F$. Thus we are again in one of the situations described in Lemma \ref{preblack} (\ref{preblack1} or \ref{preblack2}) and we can conclude. If we exclude this situation, necessarily $\al$ is a complete infimum of elements of $A$ as described in Definition \ref{completeinfimum} and Lemma \ref{prelem}.\ref{prelem1}. Hence there exists some element $\de \in (A_i \cup A_{i+1}) \cap \ds{T}{\al}$ for some $T \subseteq F$. Fix one of such $T$.

We make the following observations: 

i) Assume $G= \wh F$. Hence, $\ds{H}{\al} \neq \emptyset$ if and only if $H=F$ or $H=G$. Therefore $T=F$ and we can conclude applying one among \ref{preblack3}, \ref{preblack4} and \ref{preblack5} of Lemma \ref{preblack}. 

ii) Assume $T \subseteq G$. Hence, $\wh G \nsubseteq T$ and we can take $j \in \wh G \setminus T$. Hence $\te_j > \al_j$ and $\de_j > \al_j$ and this implies $\te \wedge \de = \te $ since $\al$ and $\te$ are consecutive. If $\de \in A_i$, we conclude that also $\te \in A_i$.
Otherwise if $\de \in A_{i+1}$, observe that $\te_l=\al_l < \de_l$ for every $l \in \wh F \subseteq G \cap \wh T$. Therefore, either $\de \gg \te$ or $\de \in \Delta^S_H(\al)$ for some $H \subseteq F$ and we conclude applying Lemma \ref{preblack} (\ref{preblack1} or \ref{preblack2}).  

(iii) Assume there exists some element $ \de \in A_i  \cap \ds{U}{\al}$ for some $\wh G \subseteq U \subseteq F$. In this case we conclude applying Theorem \ref{bianchi} to $\de, \al, \te$ to get the thesis.


By (ii) and (iii), we can restrict to assume $\de \in A_{i+1}$ and $T \nsubseteq G$.
To conclude the proof we need to show that $\dE{T}{\al} \neq \emptyset$ and then use \ref{preblack4} or \ref{preblack5} of Lemma \ref{preblack}. 

By eventually relabelling the indexes, one can assume $F=\{1,\ldots,r\}$ and
we can apply property (G2) as in Proposition \ref{propG2} to find, for every $j\in F$, a maximal set $G_j$ such that $\wh F\subseteq G_j\subseteq R_j=I\setminus\{j\}$ and $\ds{G_j}{\al}\neq \emptyset$. Recall that also $\wh F = G_1 \cap \cdots \cap G_r$ and, for every $j \neq k$, $G_j \cup G_k= I$. We can also assume that $G_r=G$ and, looking at the proof of Proposition \ref{propG2}, that $G_j\supsetneq \wh F$, otherwise we would have $G_1=G_2= \cdots= G=\wh F$ and we can conclude by the observation (i) above. Then we suppose by way of contradiction $\te \in A_{i+1}$ and for $j=1, \ldots, r$ call $\om^j$ the minimal element of $\ds{G_j}{\al}$ (hence $\te=\om^r$). It is easy to observe that, since $\ds{U}{\al}= \emptyset$ for every $U \supsetneq G_j$, then each $\om^j$ is consecutive to $\al$ and it is in $A$ since $G_j \supseteq \wh F$; thus $\om^j \in A_i \cup A_{i+1}$. Observe that if $\om^j \in A_{i+1}$, we get $(A_i \cup A_{i+1}) \cap \ds{U}{\al}= \emptyset$ for every $U \subseteq F \cap G_j$, since assuming by way of contradiction that set to be nonempty 
 and 
using the same argument used in (ii) for $\te$ (replacing $T$ by $U$, $G$ by $G_j$ and $\te$ by $\om^j$), this would imply $\om^j \in A_{i}$. This implies $T \nsubseteq F \cap G_j$ for every $j$ such that $\om^j \in A_{i+1}$. It follows that if $\om^j \in A_{i+1}$ for every $j \in\{1,\ldots r\}$, we must have $T=F$ and the proof can be completed using \ref{preblack4} and \ref{preblack5} of Lemma \ref{preblack}.

Instead, if $\om^j \in A_{i}$, we must have $\dE{\wh G_j}{\al} \neq \emptyset$ otherwise we would necessarily have $\om^j \in A_{i+1}$ by applying Theorem \ref{bianchi} to $\om^j, \al, \te$.

After a permutation of the indexes, assume $\om^j \in A_{i}$ for $j=1, \ldots, t$, with $t \leq r-1$, and $\om^j \in A_{i+1}$ for $j= t+1, \ldots, r$. As said before, $T \nsubseteq F \cap G_j$ for every $j > t$. It follows that $$ T \supseteq H:= \bigcup_{j = t+1}^{r} \wh G_j.  $$
We can conclude if we show that $\dE{U}{\al} \neq \emptyset$ for every $U$ such that $H \subseteq U \subseteq F$ (and hence in particular also for $T$). This will allow to conclude the proof by \ref{preblack4} and \ref{preblack5} of Lemma \ref{preblack} (setting the $F$ appearing in the statement of the lemma equal to $T$).
First we show that $\dE{H}{\al} \neq \emptyset$. We are going to make use of the following property which depends on the fact that each set $U$ for which $\dE{U}{\al} \neq \emptyset$ either contains $\wh G_j$ or is contained in $G_j$ (as consequence of Lemma \ref{minimidelta}.\ref{minimidelta5}),  
\begin{center}
$(\ast)$: If $U$ is a set such that $\dE{U}{\al} \neq \emptyset$ and $j \in F$, then $j \in U$ if and only if $\wh G_j \subseteq U$. 
\end{center}
 Call $\be$ the minimal element of $\dE{F}{\al}$ and recall that for every $j=1, \ldots, t$, since $\om^j \in A_{i}$, we know that $\dE{\wh G_j}{\al} \neq \emptyset$.
 Since $\wh G_1 \subsetneq F$, by applying Proposition \ref{riferimento} to $\be$ and $\dE{\wh G_1}{\al}$, we get that there exists $\he \in \dE{W_1}{\al}$ with $(F \setminus \wh G_1) \subseteq W_1 \subsetneq F$. Now, clearly $\wh G_1 \nsubseteq W_1$ and hence by $(\ast)$, $1 \not \in W_1$.
 Applying again Proposition \ref{riferimento} to $\he$ and $\dE{\wh G_2}{\al}$ we find $\he^2 \in \dE{W_2}{\al}$ with $F \setminus (\wh G_1 \cup \wh G_2) \subseteq W_1 \setminus \wh G_2 \subseteq W_2 \subsetneq W_1 \subsetneq F$. Thus, by the same argument above $1,2 \not \in W_2$.
 
Iterating the process, in the $j$-th step (for $j \leq t$), we apply Proposition \ref{riferimento} to $\he^{j-1}$ and $\dE{\wh G_j}{\al}$ in order to find $\he^j \in \dE{W_j}{\al},$ with $W_j \subsetneq F$ and $1, \ldots, j \not \in W_j$. Eventually we find a set $W_t \subsetneq F $ such that $\dE{W_t}{\al} \neq \emptyset$ and $1, \ldots, t \not \in W_t$. Hence $H=\bigcup_{j = t+1}^{r} \wh G_j = F \setminus (\bigcup_{j = 1}^{t} \wh G_j)  \subseteq W_t$. Let $j \in W_t$. If $j\not \in H$, then $j \leq t$ and this is a contradiction; therefore $W_t=H$ and $\dE{H}{\al} \neq \emptyset$.
 
 Now consider a set $U$ such that $H \subsetneq U \subseteq F$.
 For every $k \in U\setminus H$ we have $k\in\{1,\ldots t\}$ and since $k \in U\setminus G_{k}$, by $(\ast)$, $\wh G_k \subseteq U$. Hence, $\bigcup_{k \in U \setminus H} \wh G_{k}\subseteq U$. Observe that $\bigcup_{k \in U \setminus H} \wh G_{k} \cup H = U$ and by Lemma \ref{minimidelta}.\ref{minimidelta1}, $\dE{U}{\al}\neq \emptyset$ and this concludes the proof.
\end{proof}

\medskip

As application of Theorem \ref{neri} we get that the same result of Proposition \ref{infsem} holds true for the set $A$ and in particular for each level $A_i$.


\begin{prop}
\label{infap}
Let $S \subseteq \N^d$ be a good semigroup and $E \subseteq S$ be a good ideal. Set $A= S \setminus E$ and let $\bs{c}_E= (c_1, \ldots, c_d)$ be the conductor of $E$. 
Let $\al\in \N^d$ be such that $\bs{c}_E \in \dE{F}{\al}$ for some non-empty set $F \subsetneq I$. 
The following assertions are equivalent:
\begin{enumerate}
    \item \label{inf1} $\al\in A_i$.
    \item \label{inf2} $\widetilde\Delta_{\widehat{F}}(\al) \subseteq A_i$.
    \item \label{inf3} $\widetilde\Delta_{\widehat{F}}(\al)$ contains some element of the level $A_i$.
\end{enumerate}
\end{prop}
\begin{proof}
Proposition \ref{infsem} can be applied both to the good ideal $E$ and to $S$ itself. Hence it follows easily that $\al \in A$ if and only if $\widetilde\Delta_{\widehat{F}}(\al)$ contains some element of $A$ and if and only if $\widetilde\Delta_{\widehat{F}}(\al) \subseteq A$.
Hence, assuming $\widetilde\Delta_{\widehat{F}}(\al) \subseteq A$, it is sufficient to prove that two consecutive elements $\te, \om \in \widetilde\Delta_{\widehat{F}}(\al) \cup \lbrace \al \rbrace$ are in the same level. Suppose $\te < \om$ and therefore $\om \in \td{\wh F}{\te} $. Consider the element $\eps$ such that $\epsilon_j = c_j > \alpha_j = \theta_j$ for $j \in \wh F$ and $\epsilon_j = \theta_j \geq \alpha_j = c_j$ for $j \in F$. By definition of conductor, $\eps \in E$ and furthermore $\eps \in \dE{F}{\te}$. By Theorem \ref{neri}, $\te$ and $\om$ are in the same level.
\end{proof}



\section{Infinite subspaces of a good semigroup}

In this section we formally define infinite "geometric" subspaces of a good semigroup. 

For the value semigroup of an analytically irreducible ring, this definition has an algebraic interpretation which is pointed up  in \cite[Corollary 1.6]{canonical:danna}. Furthermore, these sets has been used in \cite[Proposition 1.6]{Wilf-NG} to produce a formula for the computation of the genus of a good semigroup.
Intuitively, they can be seen as lines, planes or higher dimensional subsets of $\N^d$ that are all contained in the semigroup and more particularly in a good ideal or in its complement. 

These subspaces, when of the same dimension, can be threaten as elements of a semigroup of smaller dimension. To express this concept we define a sum operation and a partial order of subspaces induced by those of the semigroup and we generalize to the set of subspaces many properties proved in the two preceding sections. The key results allowing this generalization are Propositions \ref{infsem} and \ref{infap} that in fact identify infinite subspaces with single elements having some component equal to that of the conductor. 

\begin{defi}
\rm
Let $S\subseteq \N^d$ be a good semigroup and consider a non-empty set $ U\subseteq I$ and a good ideal $E \subseteq S$. Set $A=S\setminus E$ and let $\bs{c_E}=(c_1,\ldots,c_n)$ be the conductor of $E$. For $\al \in \N^d$ such that $\alpha_j=c_j$ for all $j\in \wh U$, define $$\al(U):=\tdN{U}{\al}$$ if $U \subsetneq I$ and $\al(U):= \al$ if $U=I$. We say that $\al(U)$ is an \it$U$-subspace\rm of $\N^d$. As consequence of Propositions \ref{infsem} and \ref{infap}, it follows:
\begin{itemize}
    \item If $\al \in E$, then $\al(U) \subseteq E$ and in this case we say that it is an $U$-subspace of $E$.
    \item If $\al \in A$, then $\al(U) \subseteq A$ and in this case we say that it is an $U$-subspace of $A$. In particular, if $\al \in A_i$, the subspace $\al(U) \subseteq A_i$.
\end{itemize}
\end{defi}
We notice that, taking $E:=S$, we have also given the definition of subspaces of the good semigroup itself for every $U$. We observe the following fact: 
\begin{oss}
\label{remsubcont}
Let $\de(V)$ be a subspace and $U\supseteq V$. If $\al\in S$ and $\alpha_i=\delta_i$ for all $i\in V$, then $\al(U)\subseteq \de(V)$.
\end{oss}

In order to agree with the geometric representation of these subspaces, we define the dimension of a subspace $\al(U)$ as the cardinality $|\wh U|= d-|U|$. The subspaces of dimension zero correspond to the single elements of $S$. 
Then we call: \emph{points} the subspaces of dimension 0, \emph{lines} the subspaces of dimension 1, \emph{planes} the subspaces of dimension 2 and \emph{hyperplanes} the subspaces of dimension $d-1$.

\begin{ex}
To make a concrete example, we simplify the notation denoting, as in \cite{emb-NG}, the subspace $\al(U)$ as an element in $(\N\cup\{\infty\})^d$ such that $\beta_i=\alpha_i$ if $i\in U$ and $\beta_i=\infty$ if $i\in \wh U$.
Considering the semigroup $S$ of the Example \ref{ex3rami}
we have that:
$(2,4,\infty)$, $(\infty,3,3)$, $(\infty,3,6)$, $(\infty,3,7)$ are subspaces of $S$ having dimension $1$ (lines).\\
$(\infty,\infty,3)$, $(\infty,\infty,6)$, $(3,\infty,\infty)$, $(\infty,5,\infty)$,$(\infty,\infty,9)$, are subspaces of $S$ having dimension $2$ (planes).
\end{ex}

We denote by $E(U)$ and $A(U)$, respectively the set of $U$-subspaces of $E$ and of $A$. Furthermore, for $i\in\{1,\ldots, N\}$, we denote by $A_i(U)$ the set of $U$-subspaces contained in $A_i$.

Now we introduce the notation of $\Delta$ for $U$-subspaces (Figure \ref{fig:deltasub}). For $F \subsetneq U$, set: 
$$\dE{F}{\al(U)}:=\{\be(U)\in E(U) \hspace{0.1cm}|\hspace{0.1cm} \beta_j=\alpha_j, \text{ for any } j\in F \text{ and } \beta_j>\alpha_j \text{ if } j\in U\setminus F\}$$
$$\tdE{F}{\al(U)}:=\{\be(U)\in E(U) \hspace{0.1cm}|\hspace{0.1cm} \beta_j=\alpha_j, \text{ for any } j\in F \text{ and } \beta_j\geq\alpha_j \text{ if } j\in U\setminus F\}\setminus{\al(U)}$$

\begin{figure}[H]
\centering
 \tdplotsetmaincoords{60}{120}
\begin{tikzpicture}[tdplot_main_coords, scale=0.8, font=\footnotesize]

\draw[thick,->] (0,0,0) -- (3.5,0,0) node[anchor= east]{$1$}; 
\draw[thick,->] (0,0,0) -- (0,3.2,0) node[anchor= west]{$2$};
\draw[thick,->] (0,0,0) -- (0,0,3) node[anchor=south]{$3$};

\draw plot [mark=*, mark size=2] coordinates{(0,0,0)}  node[] at (0,0.8,1.5) {$\bs{\alpha(U)}$} node[] at (0.4,0.4,0) {$\bs{\alpha}$};


\draw[dashed] (0.5,0,0) -- (0.5,0,3);
\draw[dashed] (1,0,0) -- (1,0,3);
\draw[dashed] (1.5,0,0) -- (1.5,0,3);
\draw[dashed] (2,0,0) -- (2,0,3);
\draw[dashed] (2.5,0,0) -- (2.5,0,3);
\draw[dashed] (3,0,0) -- (3,0,3);

\end{tikzpicture}
\caption{\footnotesize{Representation of $\Delta_{2}(\al(\{1,2\}))$}}

\label{fig:deltasub}
\end{figure}
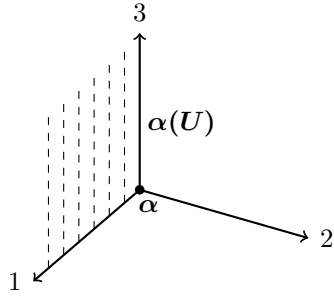


\begin{lem} Let $F\subsetneq U\subseteq I$ and let $E$ be a good ideal of a good semigroup $S$. Let $\al(U)$ be a subspace of $E$:
\label{subspaceprelem}
\begin{enumerate}
    \item \label{subspaceprelem1} $\be(U)\in \dE{F}{\al(U)}$ if and only if $\be\in \dE{F\cup \wh U}{\al}$ and $\be(U)\in \tdE{F}{\al(U)}$ if and only if $\be\in \tdE{F\cup \wh U}{\al}$ 
    \item \label{subspaceprelem2} Let $\wh F\cap U\subseteq G \subseteq U$ and $\be(U)$ a subspace of $E$.  Then $\be(U) \in \tdE{ G}{\al(U)}$ if and only if $\be \in \tdE{ G}{\al}$.
\end{enumerate}
The analogous statement hold for $A$ and $A_i$.
\end{lem}
\begin{proof}
1. If $\be(U)\in \dE{F}{\al(U)}$, then $\beta_i=\alpha_i$ for all $i\in F$ and $\beta_i>\alpha_i$ for all $i\in U\setminus F=\wh F\cap U=I\setminus(F\cup \wh U)$. Furthermore, by definition of subspace, $\beta_i=\alpha_i=c_i$ for all $i\in \wh U$. Hence we have $\be \in \dE{F\cup \wh U}{\al}$.
Conversely if $\be \in \dE{F\cup \wh U}{\al}$, then $\beta_i=\alpha_i=c_i$ for all $i\in \wh U$; hence $\be(U)$ is a subspace of $E$. Furthermore $\beta_i=\alpha_i$ for all $i\in F$, $\beta_i>\alpha_i$ for all $i\in I\setminus(F\cup \wh U)=U\setminus F$.
The other case is analogous. 
\\
2. By part \ref{subspaceprelem1}, if $\be(U)\in \tdE{G}{\al(U)}$ then $\be\in \tdE{G\cup \wh U}{\al}\subseteq \tdE{G}{\al}$, 
The converse follows immediately by the definition of $\tdE{G}{\al(U)}$ and by the fact that $\be(U)$ is a subspace of $E$.
\end{proof}

We want to extend to subspaces the sum operation of the semigroup and the partial order relation $\leq$. If $\al(U)$ and $\be(U)$ are $U$-subspace of $E$, we denote by $\bs{\sigma}\in E$ the element such that $\sigma_j=\alpha_j+\beta_j$ if $j\in U$ and $\sigma_j=c_j$ if $j\in \wh U$. By construction $\bs{\sigma}(U)$ is a subspace of $E$ and we define $\al(U)+\be(U):=\bs{\sigma}(U)$. For the order relation, we say that $\al(U)\leq \be(U)$ if $\alpha_i\leq \beta_i$ for all $i\in U$.

We also extend to subspaces Properties (G1) and (G2) of good semigroups.

\begin{prop}
Let $ U, V \subseteq I$ and let $E$ be a good ideal of a good semigroup $S$. 
\label{subspacelem}
\begin{enumerate}
\item \label{subspacelem1} If $\al(U)$ and $\be(V)$ are subspace of $E$, then $\al(U)\wedge\be(V):=(\al \wedge \be)(U\cup V)$ is still a subspace of $E$.
\item \label{subspacelem2} Let $F \subsetneq U$. If $\al(U)\in E(U)$ and $\dE{F}{\al(U)}\neq \emptyset$, then $\tdE{U\setminus F}{\al(U)}\neq \emptyset$.
\end{enumerate}
\end{prop}
\begin{proof} 
1. It is sufficient to notice that, if $i\in I\setminus (U\cup V)=\wh U\cap \wh V$, then $\alpha_i=\beta_i=c_i$.\\
2. If $\al(U)\in E(U)$ then $\al\in E$. Since $\dE{F}{\al(U)}\neq \emptyset$, by Lemma \ref{subspaceprelem}.\ref{subspaceprelem1}, it follows $\tdE{F \cup \wh U}{\al}\neq \emptyset$. By Property (G2) applied on $E$ we have $\tdE{\wh F\cap U}{\al}\neq \emptyset$; then by Lemma \ref{subspaceprelem}.\ref{subspaceprelem2}, $\tdE{\wh F\cap U}{\al(U)}=\tdE{U\setminus F}{\al(U)}\neq \emptyset$. 
\end{proof}

We can also generalize in a natural way the definition of complete infimum for subspaces.

\begin{defi} \rm Given a good semigroup $S\subseteq \N^d$, 
we say that the subspace $\al(U)$ is a \it complete infimum \rm if there exists $\be^{1}(U),\ldots,\be^{(r)}(U)\in S(U)$, with $r\geq 2$, satisfying the following properties:
\begin{enumerate}
\item $\be^{(j)}(U) \in \ds{F_j}{\al(U)}$ for some non-empty set $F_j \subsetneq U$.
\item For every choice of $j,k \in \lbrace 1, \ldots, r \rbrace$, $\al(U) = \be^{(j)}(U) \wedge \be^{(k)}(U) $.
\item $\bigcap_{k=1}^r {F_k}= \emptyset$.
\end{enumerate}
In this case we write $\al(U)=\be^{(1)}(U)\wt\be^{(2)}(U)\cdots\wt\be^{(r)}(U)$.
\end{defi}

The generalization of Proposition \ref{propG2} for subspaces will be particularly useful in the following.

\begin{prop}
\label{propG2sub} Given $F \subsetneq U$,
if $\be(U)\in \dE{F}{\al(U)}$, there exist $\be^{(1)}(U), \ldots, \be^{(r)}(U)$ with $1\leq r\leq |F|$, such that:
\[\al(U)=\be(U)\wt \be^{(1)}(U) \wt \be^{(2)}(U)\wt \cdots \wt \be^{(r)}(U) \]
and, for $i=1,\ldots r$, $\be^{(i)}(U)\in\dE{G_i}{\al(U)}$ with $G_i\supseteq  \wh F\cap U$ and $G_1\cap G_2\cap \cdots\cap G_r=\wh F \cap U$. 
\end{prop}
\begin{proof}
Since $\be(U)$ is a subspace of $E$, $\dE{U}{\be}\neq \emptyset$. Let us consider $\be'\in \dE{U}{\be}$. We notice that $\beta'_i=\beta_i=\alpha_i$ if $i\in F$,  $\beta'_i=\beta_i>\alpha_i$ if $i\in U\setminus F$, and $\beta'_i>\beta_i=\alpha_i=c_i$ if $i\in \wh U$. Hence $\be'\in \dE{F}{\al}$. By Proposition \ref{propG2}, applied on $\be'$ and $\al$, there exists $\be^{(1)}, \ldots, \be^{(r)}$ (with $1\leq r\leq |F|$), such that 
$\al=\be' \wt \be^{(1)} \wt \be^{(2)}\wt \cdots \wt \be^{(r)}$. It follows that for $i=1,\ldots r$, $\be^{(i)}\in\dE{G'_i}{\al}$ and $G'_1\cap G'_2\cap \cdots \cap G'_r=\wh F \supsetneq \wh U$. \\ 
For every $i$, we notice that $\beta^{(i)}_j=\alpha_j=c_j$ if $j\in\wh U$, $\beta^{(i)}_j=\alpha_j$ if $j\in U\cap G'_i$, and $\beta^{(i)}_j=\alpha_j=c_j$ if $j\in U\cap \wh G'_i$. This implies $\be^{(i)}(U)\in \dE{U\cap G'_i}{\al(U)}$. Moreover,
$$\bigcap_{i=1}^r(U\cap G'_i)=U\cap(\bigcap_{i=1}^r G'_i)=U\cap \wh F $$
and we conclude setting $G_i:= U \cap G'_i$.
\end{proof}

Now we extend some other properties proved in the previous sections. The proofs are analogous to those seen for Proposition \ref{usodiG2} and Lemma \ref{minimidelta}.

\begin{lem} 
\label{varsubspaces}
Let $S$ be a good semigroup and $U\subseteq I$, $U\neq \emptyset$.
\begin{enumerate}
    \item \label{usodiG2sub} Assume $\al(U) \in E(U)$, $F \subsetneq U$ and there exists a set $H\subsetneq U$ of cardinality $|U|-1$ containing $F$ and such that $\dE{G}{\al(U)} = \emptyset$ for every $F \subseteq G \subseteq H$. Then $\dE{\widehat{F}}{\al(U)} = \emptyset$.
    \item \label{minimidelta1sub}If $\be(U) \in \dE{F}{\al(U)}$ and $\te(U) \in \dE{G}{\al(U)}$, then $\be(U) \wedge \te(U) \in \dE{F \cup G}{\al}$ if $F \cup G \subsetneq U$ and $\be(U) \wedge \te(U) = \al(U)$ if $F \cup G = U$.
    \item \label{minimidelta2sub}Let $\be(U) \in \dE{F}{\al(U)}$ be consecutive to $\al(U)$ in $E$. Then $\dE{H}{\al(U)}= \emptyset$ for every set $H$ such that $ F \subsetneq H \subsetneq U$.
\end{enumerate}
\end{lem}

It is also possible to rephrase Theorems \ref{bianchi} and \ref{neri} for subspaces.

\begin{teo}
\label{bianchienerisub}
Let $S$ be a good semigroup, $F\subsetneq U \subseteq I$, $E\subseteq S$ a good ideal and $A=S\setminus E$. Consider $\al(U)\in S(U)$, $\be(U) \in \ds{F}{\al(U)}$ and $\te(U)\in \ds{G}{\al(U)}$ with $(\wh F \cap U) \subseteq G \subsetneq U$ and assume $\te(U)$ and $\al(U)$ to be consecutive in $S(U)$.
\begin{enumerate}
\item \label{bianchisub} 
Assume $\be(U) \in A_i(U)$ and $\ds{\wh F\cap U}{\al(U)}\subseteq A(U)$. 
\begin{itemize} 
\item If $\td{G}{\al(U)}\subseteq A(U)$, then $\al(U)\in A_h(U)$ with $h \leq i$.
\item If $\al(U)\in A(U)$ and $\td{F}{\al(U)}\subseteq A_i(U)$ then $\al(U)\in A_h(U)$ with $h<i$.
\end{itemize}
\item \label{nerisub} 
Assume $\al(U) \in A_i(U)$. If $\dE{F}{\al(U)}\neq \emptyset$, then $\te(U)\in A_i(U)$.
\end{enumerate}
\end{teo}

\begin{proof}
1. Directly by definition of subspace we have $\al \in S$ and $\be\in A_i$.  By Lemma \ref{subspaceprelem}.\ref{subspaceprelem1}, we also have $\be\in\ds{F\cup \wh U}{\al}$, $\te\in \ds{G}{\al}$ and $\te$,$\al$ consecutive in $S$.\\
We prove that $\ds{U\setminus F}{\al}\subseteq A$. If $\he \in \ds{U\setminus F}{\al}$, we can consider $\de \in \N^d$ such that: $\delta_i=c_i$ for $i\in \wh U$, $\delta_i=\eta_i$ for $i\in U$. Clearly $\de(U)$ is an $U$-subspace by definition and $\de(U)\in \ds{\wh F\cap U}{\al(U)}\subseteq A(U)$, hence, by Proposition \ref{infap}, $\de, \he \in A$. \\
Assuming $\td{G}{\al(U)}\subseteq A(U)$, then, by Lemma \ref{subspaceprelem}.\ref{subspaceprelem2}, $\td{G}{\al}\subseteq A$ and $\al\in A$. Hence, by Theorem \ref{bianchi}, $\te \in A_h$ with $h\leq i$ and clearly $\te(U) \in A_h(U)$.
Assuming also $\al(U)\in A(U)$, we get $\al\in A$. Furthermore, again by Lemma \ref{subspaceprelem}.\ref{subspaceprelem2}, $\td{\wh F \cap G}{\al}\subseteq A$. Hence, by Theorem \ref{bianchi}, $\al \in A_h$ with $h\leq i$ and we conclude in the same way. 

2. The assumption implies $\al\in A_i$. Since $\dE{F}{\al(U)}\neq \emptyset$, by Lemma \ref{subspaceprelem}.\ref{subspaceprelem1},  $\dE{F\cup \wh U}{\al}\neq \emptyset$. Since $\te(U)\in \tdE{G}{\al(U)}$, by Lemma \ref{subspaceprelem}.\ref{subspaceprelem2}, $\te\in \tdE{G}{\al}$ and $\te $ is consecutive to $\al$ in $S$. By Theorem \ref{neri}, $\te\in A_i$ and we can conclude that $\te(U)\in A_i(U)$.
\end{proof}

\section{The number of levels of the Ap\'{e}ry set}

In this section we compute the number of levels of the complement of a good ideal generated by a single element. For a good semigroup $S \subseteq \N^d$ and $\om \in S$, we consider the ideal $E= \om + S$. Its complement $A = S \setminus E$ is classically called the Ap\'{e}ry set of $S$ with respect to $\om$ and denoted by $\Ap(S, \om)$. Often the Ap\'{e}ry set is considered with respect to the minimal nonzero element $\bs{e}$ but for the purpose of this section it is enough to consider it with respect to any given element $\om$. 
In Theorem \ref{main} we extend to the $d$-branches case both \cite[Theorem 3]{DAGuMi} and \cite[Theorem 5]{DAGuMi}. In particular we prove that the number of levels of the partition of $\Ap(S, \om)$ is equal to the sum of the components of $\om$. This provides an alternative proof of \cite[Theorem 3]{DAGuMi} in the case $d=2$. We recall the content of \cite[Theorem 5]{DAGuMi}, with the slight change of considering $\Ap(S, \om)$ for a non-specific $\om \in S$. 

\begin{prop} \rm (\cite[Theorem 5]{DAGuMi}) \it
\label{infiniti}
Let $S \subseteq S_1 \times S_2 \subseteq \N^2$ be a local good semigroup, let $\om=(w_1, w_2) \in S$. Let $ A:=\Ap(S,\om)=\bigcup_{i=1}^N A_i $ be the Ap\'{e}ry set of $S$ with respect to $\om$. Assume $w_1 \geq w_2$. Then: \\
 The levels $A_N, A_{N-1}, \dots, A_{N-w_1+1}$ are infinite and in particular there exist $s_1, \ldots, s_{w_1} \in S_1$ such that for every $i=1,\ldots,w_1 $, the set $\Delta^S_1(s_i,0)$ is infinite and eventually contained in the level $A_{N-i+1}$. \\
 Furthermore, all the levels $A_i$ with $i < N-w_1+1$ are finite. 
 If $w_1 \leq w_2$ the correspondent analogous conditions hold.
\end{prop} 

\begin{proof} One can use the same proof done in \cite[Theorem 5]{DAGuMi} in the case $\om = \e$.
\end{proof}

Before to prove the main theorem we need another result involving next definition. In the case $d=3$ this definition is explained in Figure \ref{fig:notationsH}.
\begin{defi}
\label{insiemone} 
For $U \subseteq I$ and 
for a subspace $\al(U)$ and $k \in U$ we set $$ H_k(\al(U)):= \lbrace \be(U) \subseteq S \hspace{0.1cm}|\hspace{0.1cm} \alpha_k = \beta_k \rbrace.$$
In particular, for $\al \in S$ we set $$ H_k(\al):= H_k(\al(I))= \lbrace \al \in S \hspace{0.1cm}|\hspace{0.1cm} \alpha_k = \beta_k \rbrace. $$
\end{defi}

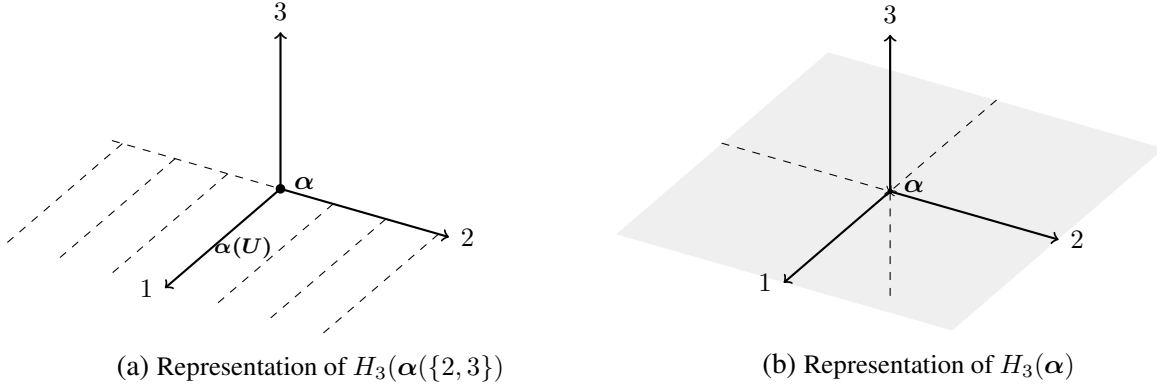
\begin{figure}[H]
\centering

\begin{subfigure}{.5\textwidth}
 \tdplotsetmaincoords{60}{120}
\begin{tikzpicture}[tdplot_main_coords, scale=0.8, font=\footnotesize]

            
\draw[thick,->] (0,0,0) -- (3.8,0,0) node[anchor= east]{$1$};
\draw[thick,->] (0,0,0) -- (0,3.2,0) node[anchor= west]{$2$};
\draw[dashed] (0,-3.2,0) -- (0,0,0);
\draw[thick,->] (0,0,0) -- (0,0,3) node[anchor=south]{$3$};

\draw plot [mark=*, mark size=2] coordinates{(0,0,0)}  node at (-0.35,0.25,0) {$\bs{\alpha}$}  node[font=\scriptsize] at (2,0.45,0) {$\bs{\alpha(U)}$};

\draw[dashed] (0,1,0) -- (3.8,1,0);
\draw[dashed] (0,2,0) -- (3.8,2,0);
\draw[dashed] (0,3,0) -- (3.8,3,0);
\draw[dashed] (0,-1,0) -- (3.8,-1,0);
\draw[dashed] (0,-2,0) -- (3.8,-2,0);
\draw[dashed] (0,-3,0) -- (3.8,-3,0);
\end{tikzpicture}
\caption{\footnotesize{Representation of $H_{3}(\al(\{2,3\})$}}
\end{subfigure}%
\begin{subfigure}{.5\textwidth}
 \tdplotsetmaincoords{60}{120}
\begin{tikzpicture}[tdplot_main_coords, scale=0.8, font=\footnotesize]

\filldraw[draw=white, fill=black!20, fill opacity=0.3]         
            (3.5,3.2,0)
            -- (3.5,-3.2,0)
            -- (-3.5,-3.2,0)
            -- (-3.5,3.2,0)
            -- cycle; 
            
\draw[thick,->] (0,0,0) -- (3.5,0,0) node[anchor= east]{$1$};
\draw[dashed,->] (-3.5,0,0) -- (0,0,0);
\draw[thick,->] (0,0,0) -- (0,3.2,0) node[anchor= west]{$2$};
\draw[dashed,->] (0,-3.2,0) -- (0,0,0);
\draw[thick,->] (0,0,0) -- (0,0,3)
node[anchor=south]{$3$};
\draw[dashed,->] (0,0,-2) -- (0,0,0);

\draw plot [mark=*, mark size=1] coordinates{(0,0,0)}  node at (-0.35,0.25,0) {$\bs{\alpha}$};

\end{tikzpicture}
 \caption{\footnotesize{Representation of $H_{3}(\al)$}}
\end{subfigure}%
\caption{\footnotesize{A graphical representation of Definition \ref{insiemone}.}}
\label{fig:notationsH}
\end{figure}

\begin{prop}
\label{cambioelemento}
Let $S$ be a good semigroup, $E \subseteq S$ a good ideal and $A= S \setminus E$. Take $U \subseteq I$. Suppose that $\al(U)$ is a subspace of maximal dimension contained in the level $A_i$ (equivalently $A_i$ does not contain subspaces of dimension $\geq d-|U| + 1$). Then, fixed $k \in U$, there exists $\be(U) \in A_i(U)$ such that $H_k(\be(U)) \subseteq A(U)$ and $\beta_k\geq \alpha_k$.
\end{prop}
 
 \begin{proof}
We first show that, if $H_k(\al(U)) \nsubseteq A(U)$, we can find $\te(U) \in A_i(U)$ such that $\theta_k > \alpha_k$.
 Consider the case in which $\tdE{k}{\al(U)} \neq \emptyset$
 and therefore assume $\dE{F}{\al(U)} \neq \emptyset$ for some $U\supsetneq F \supseteq \{ k \} $. By Proposition \ref{propG2sub}, there exists $\te(U) \in \ds{G}{\al(U)}$ consecutive to $\al(U)$ with $U\supsetneq G \supseteq \wh F\cap U$. Furthermore we can choose $G$ such that $k \not \in G$ and thus $\theta_k > \alpha_k$. 
 By Theorem \ref{bianchienerisub}.\ref{nerisub}, $\te(U) \in A_i(U)$ and this case is complete. 
 
 Hence, assume on the other hand $\tdE{k}{\al(U)} = \emptyset$, but there exists $\bs{\tau}(U) \in E(U)$ such that $\tau_k = \alpha_k$ and $\tau_l < \alpha_l$ for some $l \neq k$. From this assumption, it follows that there exists $\he(U) < \al(U)$ such that $\eta_k= \alpha_k$ and $\he(U)$ and $\al(U)$ are consecutive (in the case $\tau_j > \alpha_j$ for some $j$ one can consider $\bs{\tau}(U) \wedge \al(U)$). Say that $\al(U) \in \ds{F}{\he(U)} $ for $F \supseteq \{ k \} $. Now, if $\he(U) \in A_i(U)$ we can replace $\al(U)$ by $\he(U)$ and iterate the process, considering also, if needed, the case $\dE{F}{\he(U)} \neq \emptyset$ as done above. Thus we can restrict to the case in which $\he(U) \in E(U)$ or $\he(U) \in A_{i-1}(U)$.
 
 In any case, using Proposition \ref{propG2sub} as before, we can find $\te(U) \in \ds{G}{\he(U)} $ consecutive to $\he(U)$ with $G \supseteq \wh F$ and $k \not \in G$ (in particular $\theta_k> \alpha_k$). We need to prove that $\te(U) \in A_i(U)$. 
 
 Since $\he(U)$ and $\al(U)$ are consecutive, by Lemma \ref{varsubspaces}.\ref{minimidelta2sub}, $\ds{H}{\he(U)} = \emptyset$ for every $H$ with $U\supsetneq H \supsetneq F$ . Furthermore, if $H \subseteq F$ and $k \in H$, one can easily observe that $\ds{H}{\he(U)} \subseteq \lbrace \al(U) \rbrace \cup \td{k}{\al(U)}$ and thus  $\dE{H}{\he(U)} = \emptyset$. In particular $\dE{F}{\he(U)}$ and $ \dE{\wh G \cap U}{\he(U)} $ are both empty (notice that $\wh G\cap U \subseteq F$). Now, if $\he(U) \in E(U)$, by Proposition \ref{varsubspaces}.\ref{usodiG2sub}, this implies $\dE{G}{\he(U)} = \emptyset$ and $\dE{\wh F}{\he(U)} = \emptyset$. If instead $\he(U) \in A_{i-1}(U)$, we get the same result, observing that if $\dE{G}{\he}\cup\dE{\wh F\cap U}{\he} \neq \emptyset$, by Theorem \ref{bianchienerisub}.\ref{nerisub} we would have $\al(U) \in A_{i-1}(U)$, a contradiction.
 
 By all these facts, it follows that $\te(U) \in A(U)$, moreover, since $\te(U)$, $\he(U)$ and $\al(U)$ are consecutive and the assumptions of Theorem \ref{bianchienerisub}.\ref{bianchisub} are satisfied, we get $\te(U) \in A_i(U)$.
 
 Now if $H_k(\te) \subseteq A(U)$ we are done, otherwise we iterate the process creating a sequence of elements, all in the level $A_i(U)$, having strictly increasing $k$-th component. If this process does not stop, we reach an element $\de(U) \in A_i(U)$ such that $\de_k \geq (\bs{c}_{E})_k$. Set $\de^{\prime} = \de \wedge \bs{c}_{E}$ and observe that $\bs{c}_{E} \in \ds{V}{\de^{\prime}}$ for some $V$ containing $k$ and $V\supsetneq \wh U$. 
 By Proposition \ref{infap}, the subspace $\de^{\prime}(\wh V) \in A_i(\wh V) $. This is a contradiction since we assumed $A_i$ to not contain subspaces of dimension larger than $d-|U|$. 
 It follows that the process described above must instead stop to some element $\be(U) \in A_i(U)$ such that $H_k(\be(U)) \subseteq A(U)$ and $\beta_k\geq \alpha_k$.
\end{proof}

Now we have got enough tools to prove the main theorem. 

\begin{teo}
\label{main}
Let $S \subseteq \N^d$ be a good semigroup, let $\om=(w_1,\ldots,w_d) \in S$ be a nonzero element, and assume without restrictions $w_1\geq w_2 \geq \cdots \geq w_d$. For $k=1, \ldots,d$, set  $$s_k=\sum_{j=1}^k w_j.$$ Let $E=\om + S$, $A= S \setminus E$ and write the partition $A = \bigcup_{i=1}^N A_i$. 
Then, the level $A_i$ contains some subspace of dimension $d-k$ if and only if $i\geq N-s_k+1$.
In particular, \[N= s_d = \sum_{i=1}^d w_i.\]
\end{teo}

\begin{proof}
For $l=1, \ldots, d$ denote by $\eps^l$ the element of $\N^d$ having $l$-th component equal to $1$ and all the others equal to zero. We prove the theorem by induction on $k \geq 1$.

\bf Base case: \rm For $k=1$, we have to show that the levels containing some hyperplane are exactly the levels $A_N, A_{N-1}, \ldots, A_{N-w_1+1}.$ 

For $j=1,\ldots, w_1$, set $\be^{j}= \bs{c}_E + j\eps^1 = (c_1+ j, c_2, \ldots, c_d)$ and observe that the hyperplanes $\be^j(1) \in E(1)$. For each of these $\be^{j}$, 
we show that there exists a unique integer $m_j \geq 1$ such that $\al^j := \be^{j} - m_jw_1 \eps^1 \in A $. 
In fact, fixed $j$, we consider the quotient $q^j$ and the remainder $r^j$ of the division of $c_1+j$ by $w_1$.
Define the element $\bs{\tau}^j\in \N^d$ such that $\tau_1=\beta_1^j$ and $\tau_k=c_k+q^jw_k$ for $k \neq 1$. Since $\be(1)\in E(1)$, it follows that $\bs{\tau^j}\in E$. Moreover, since $E=S+\om$, there exists $m_j\in \N$ such that $\bs{\tau}^j-m_j\om\in A$. Looking at first component, since $E=S\setminus A$, we have
$$q_jw_1\geq m_jw_1\geq m_jw_k.$$
Hence $\bs{\tau}^j-m_j\om$ has the first component equal to $\beta_1^j-m_jw_1$ and the other components are strictly larger than the components of conductor. By Proposition \ref{infap}, we have $\al^j := \be^{j} - m_jw_1 \eps^1 \in A $ and equivalently $\al^j(1) \in A(1)$.
 
Furthermore, for $s \geq 1$, the first component of each element of the form $\bs{c}_E - s\eps^1$ is congruent modulo $w_1$ to the first component of one of the $\be^{j}$ and therefore $\al^1(1), \ldots, \al^{w_1}(1)$ are exactly the only $\{1\}$-hyperplanes 
contained in $A$.
 
By Proposition \ref{infap}, each set $\al^j(1)$ is contained in a unique level. Since $A_N = \Delta(\ga_E)$, we know that $\al^{w_1}(1)\subseteq A_N$.
Now we can relabel the indexes of the hyperplanes $\al^1(1),\ldots,\al^{w_1}(1)$ in the way that $\al_1^1<\cdots<\al_1^{w_1}$.

For $j < N$, given $ \de \in \al^j(1)$, there exists always $\te \in \al^{j+1}(1)$ such that $\de \ll \te$ and hence, by Proposition \ref{infap}, $\al^j(1)$ and $\al^{j+1}(1)$ are contained in two different levels. Such two levels are consecutive since, by construction, any subspace $\be(1)$ such that $\al^j(1) < \be(1) < \al^{j+1}(1)$ is not contained in $A$. This implies that $\al^j(1) \subseteq A_{N - w_1 +j}(1)$ and hence the levels $A_N, A_{N-1}, \ldots, A_{N-w_1+1}$ contain at least one $\{1\}$-hyperplane.
 
Using the same argument for $l \neq 1$, we get that the levels $A_N, A_{N-1}, \ldots, A_{N-w_l+1}$ contain some hyperplane. But, since $w_1 \geq w_l$, the level $A_i$ contains some hyperplane if and only if $i \geq N-w_1+1$.

\bf Inductive step: \rm By induction we can assume $k \geq 2$ and the thesis true for $k-1$. Hence the levels $A_N, \ldots, A_{N-s_{k-1}+1}$ are the only levels containing a subspace of dimension $d-k+1$ and,  by Remark \ref{remsubcont}, clearly also contain subspaces of any lower dimension. Call $D:=N-s_{k-1}+1$. We claim that for $i < D$ the levels $A_i$ contains a subspace of dimension $d-k$ if and only if $i \geq N-s_{k}+1 = D - w_k$. 

Set $U=\lbrace 1,\ldots, k \rbrace$ and $V=\lbrace 1,\ldots, k-1 \rbrace$. By inductive hypothesis, there exists $\de \in S$ such that $\de(V)$ is the minimal $V$-subspace contained in $A_D$ (it has dimension $d-k+1$). Hence, by Remark \ref{remsubcont}, there are clearly infinitely many $U$-subspaces contained in the level $A_D$. Among them, for $j=1,\ldots, w_k$, consider 
$\de^j(U)\in A_D(U)$ minimal with respect to the property of having $\de^j_k \equiv j $ mod $w_k$.
For each $j$, we show that  $\tdE{k}{\de^j(U)} \neq \emptyset$. Indeed, after fixing $\de^j(U)$, using the fact that there are infinitely many $U$-subspaces contained in $\de(V)$, we can find $\de^{\prime}(U) \in A_D(U)$ such that 
$\delta^{\prime}_k > \delta^j_k$ (observe that since they are in the same level necessarily $\delta^{\prime}_h = \delta^j_h$ for some $h < k$).
  
Now, if $\tdE{k}{\de^j(U)} = \emptyset$, applying Proposition \ref{propG2sub} to $\de^j(U)$ and $\de^{\prime}(U)$, we can write $$ \de^j(U) = \de^{\prime}(U) \wt \be^1(U) \wt \cdots \wt \be^r(U) $$ where $\be^l(U) \in \td{k}{\de^j(U)} \subseteq A(U)$ and we may assume them to be consecutive to $\de^j(U)$ for all $j\in{1,\ldots r}$. By Theorem \ref{bianchienerisub}.\ref{bianchisub}, for every $l$, $\be^l(U) \in A_D(U)$ implying that $\de^j(U)$ has to be in a lower level, a contradiction (for a graphical representation see Figure \ref{fig:maintheorema}).
  
Hence, we can set $\bs{\tau}^j(U)$ to be a minimal element in $\tdE{k}{\de^j(U)}$. We define $\overline{\om}$ such that $\overline{\omega_i}=\omega_i$ if $i\in U$ and $\overline{\omega_i}=c_i$ otherwise, and, starting from $\bs{\tau}^j(U)$ and subtracting multiples of 
$\overline{\om}(U)$, as we have seen above, we find a unique $m_j \geq 1$ such that $ \al^j(U) := \bs{\tau}^j(U) - m_j \overline{\om}(U) \in A(U) $. 
Now we consider the set $H_k(\al^j(U))$ defined in Definition \ref{insiemone}. In the case this set contains some subspace of $E$, starting by one them and subtracting multiples of $\overline{\om}(U)$, we can repeat the process and, after changing names, we can finally assume to have a collection of subspaces $\al^1(U), \ldots, \al^{w_k}(U) \in A(U)$ such that for every $j$, $\alpha^j_k \equiv \delta^j_k \equiv j $ mod $w_k$ and $H_k(\al^j(U)) \subseteq A(U)$. Using this last condition, we can make a further change and assume $\al^j(U)= \min H_k(\al^j(U))$, which is well-defined by Proposition \ref{subspacelem}.\ref{subspacelem1} (see Figure \ref{fig:maintheoremb}).
\begin{figure}[H]
\begin{subfigure}{.53\textwidth}
  \tdplotsetmaincoords{60}{120}
\begin{tikzpicture}[tdplot_main_coords, scale=1.3]

\filldraw[draw=white, fill=black!20, fill opacity=0.9]         
            (0,0,0)
            -- (3.5,0,0)
            -- (3.5,0,2)
            -- (0,0,2)
            -- cycle;

\draw[dashed,->] (-3.5,0,0) -- (3.5,0,0) node[anchor= east]{\scalebox{.7}{$1$}}; 
\draw[dashed,->] (0,0,0) -- (0,3.2,0) node[anchor= west]{\scalebox{.7}{$2$}};
\draw[dashed,->] (0,0,0) -- (0,0,3) node[anchor=south]{\scalebox{.7}{$3$}};
\draw node at (1.1,0,1) {\scalebox{1.0}{$\bs{\delta^{\prime}(U)}$}};
\draw node at (-2.8,0.1,1) {\scalebox{1.0}{$\bs{\delta^{j}(U)}$}};
\draw node at (-2,2.6,1.3) {\scalebox{1.0}{$\bs{\beta^{l}(U)}$}};
\draw [fill=white] (-2,0,2) circle[radius= 0.2 em];
\draw [fill=white] (2,0,2) circle[radius= 0.2 em];
\draw [fill=white] (0,0,2) circle[radius= 0.2 em];
\draw [fill=white] (-2,2,2) circle[radius= 0.2 em];
\draw[] (0,0,0) -- (0,0,2); 
\draw[] (-2,0,0) -- (-2,0,2); 
\draw[] (-2,2,0) -- (-2,2,2); 
\draw[] (2,0,0) -- (2,0,2); 

\end{tikzpicture}

  \caption{}
  \label{fig:maintheorema}
\end{subfigure}%
\begin{subfigure}{.5\textwidth}
\begin{tikzpicture}[scale=0.8, font=\footnotesize]
	
	\draw[thick,->] (-1,0) -- (8,0) node[anchor= north]{$1$}; 
    \draw[thick,->] (0,-1) -- (0,8) node[anchor= east]{$2$};
	\draw [fill=white] (0,0) circle[radius= 0.5 em];
	\draw node at (0,0) {$3$};

	\draw[thick] (5,5) -- (5,8); 
	\draw node at (5.8,6.5) {\scalebox{1.20}{$\bs{\delta^{j}(V)}$}};
	
    \draw [fill=white] (5,4) circle[radius= 0.3 em];
	\draw node at (4.1,4) {\scalebox{1.20}{$\bs{\delta^{j}(U)}$}};
	\draw [fill=black] (6.5,4) circle[radius= 0.3 em];
	\draw node at (7.5,4) {\scalebox{1.20}{$\bs{\tau^{j}(U)}$}};
	
	\draw [fill=white] (4.5,3) circle[radius= 0.3 em];
	\draw [fill=black] (3.5,3) circle[radius= 0.3 em];
	\draw[dashed,thick,->] (4.5,3) -- (5,3.25); 
	\draw[dashed,thick,->] (5,3.25) -- (5.5,3.5); 
	\draw[dashed,thick,->] (5.5,3.5) -- (6,3.75); 
	\draw[dashed,thick,->] (6,3.75) -- (6.5,4); 
	
	\draw [fill=white] (2,2.25) circle[radius= 0.3 em];
	\draw [fill=black] (3,2.25) circle[radius= 0.3 em];
    \draw[dashed,thick,->] (2,2.25) -- (2.5,2.5); 
	\draw[dashed,thick,->] (2.5,2.5) -- (3,2.75); 
	\draw[dashed,thick,->] (3,2.75) -- (3.5,3); 
	
	\draw [fill=white] (1,1.25) circle[radius= 0.3 em];
	\draw node at (1,0.8) {\scalebox{1.20}{$\bs{\alpha^{j}(U)}$}};
    \draw[dashed,thick,->] (1,1.25) -- (1.5,1.5); 
	\draw[dashed,thick,->] (1.5,1.5) -- (2,1.75); 
	\draw[dashed,thick,->] (2,1.75) -- (2.5,2); 
    \draw[dashed,thick,->] (2.5,2) -- (3,2.25); 

    \draw (-1,1.25) -- (8,1.25); 
    \draw node at (6,0.8) {\scalebox{1.20}{$H_2(\bs{\alpha^{j}(U)})$}};
	\end{tikzpicture}
   \caption{}
   \label{fig:maintheoremb}
\end{subfigure}%
\caption{\footnotesize{\ref{fig:maintheorema}: we have $d=3$, $U=\{1,2\}$, $\de^j(U),\de^{\prime}(U),\be^l(U)$ are lines. \ref{fig:maintheoremb}: this is a perspective from "above" of the case $d=3$, $U=\{1,2\}$, $V=\{1\}$. In this case $\de^{\prime}(V)$ is a plane contained in $A$; $\de^j(U),\bs{\tau}^j(U),\al^j(U)$ are lines.}}
\label{fig:maintheorem}
\end{figure}
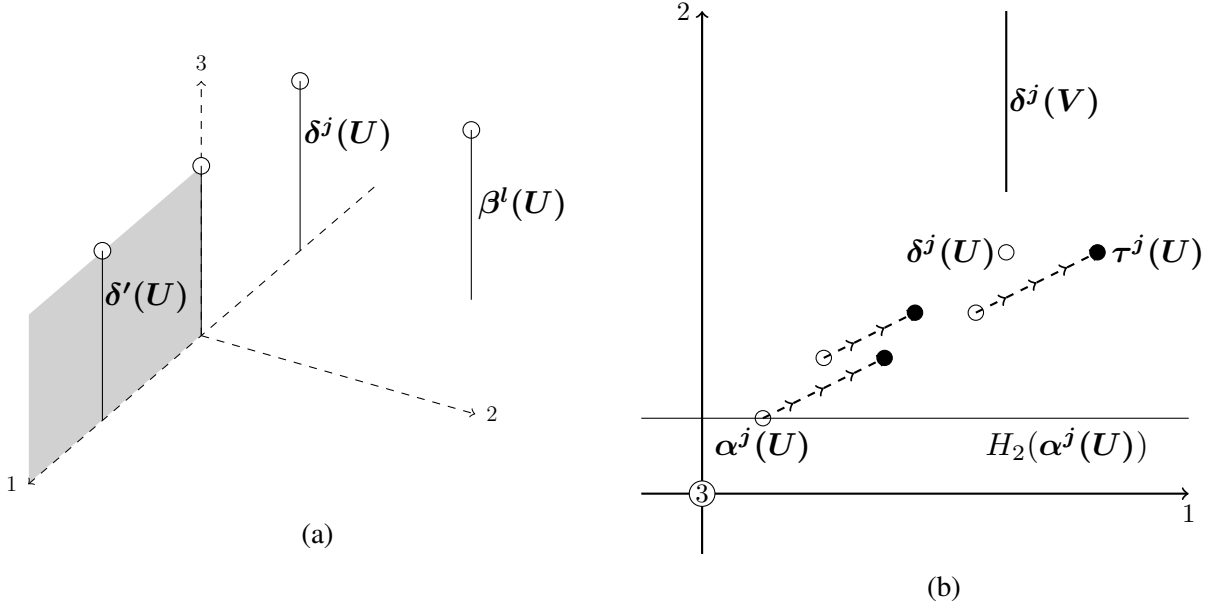
After this change, it follows that $\al^1(U), \ldots, \al^{w_k}(U)$ are totally ordered with respect to the standard order $\leq$ and, by relabeling the indexes, suppose $$\al^1(U) > \cdots > \al^{w_k}(U).$$
  
Now, notice that for every $j$, the level of $\al^j(U)$ has to be strictly lower than $D$ since $\de^j(U)$ has been chosen to be the minimal in $A_D$ having $k$-th component congruent to $j$ modulo $w_j$.
  
We prove that $\al^1(U), \ldots, \al^{w_k}(U) $ are pairwise contained in different levels. 
It is enough to show this for $\al^j(U)$ and $\al^{j+1}(U)$. 
If $\al^j(U) \gg \al^{j+1}(U)$ this is clear, otherwise we must have $\al^j \in \ds{F}{\al^{j+1}}$ for some $F \subseteq U$ such that $k \not \in F$. Using Proposition \ref{propG2sub} as above together with the fact that $\td{k}{\al^{j+1}(U)} \subseteq H_k(\al^{j+1}(U)) \subseteq A(U)$, we can express $\al^{j+1}(U)$ as complete infimum of $\al^j(U)$ and other elements in $\td{k}{\al^{j+1}(U)}$, consecutive to $\al^{j+1}(U)$. By Theorem \ref{bianchienerisub}.\ref{bianchisub}, it follows that the level of $\al^{j+1}(U)$ has to be strictly smaller than the level of $\al^j(U)$. 

Now we prove that if an $U$-subspace is contained in the level $A_i$ with $i<D$, then $i$ is larger or equal than the level of $\al^{w_k}(U)$.   
Let $\al(U)$ be any other $U$-subspace contained in $A_i$ for some $i < D$. By Proposition \ref{cambioelemento}, there exists $\al^{\prime}(U) \in A_i(U)$ such that $\alpha_k \leq \alpha^{\prime}_k$ and $H_k(\al^{\prime}(U)) \subseteq A(U)$. One necessarily must have $H_k(\al^{\prime}(U))= H_k(\al^{j}(U))$ for the $j$ such that $\alpha^{\prime}_k \equiv \alpha^j_k$ mod $w_j$, since if this was not true, by summing multiples of $\overline{\om}(U)$, we would have some subspace of $E$ contained in one set among $H_k(\al^{\prime}(U))$ and $ H_k(\al^{j}(U))$ and this is not possible. From this, since $$ \al^{w_k}(U) \leq \al^{j}(U) = \min H_k(\al^j(U)) \leq \al^{\prime}(U),$$
it follows that $i$ is not lower than the level of $\al^{w_k}$.

To describe all the possible levels of the $U$-subspaces it remains to show that for $j=1, \ldots, w_k$, we have $\al^j(U) \subseteq A_{D-j}$ and in particular $\al^{w_k}(U) \subseteq A_{D-w_k}= A_{N - s_k + 1}$. Use induction on $j$. For $j=1$, assume by way of contradiction that $\al^1(U) \in A_i(U)$ for some $i < D-1$. Hence, by Lemma \ref{prelem}.\ref{prelem1}, there exists $\be(U) \in A_{D-1}(U)$ such that $\be(U) \geq \al^1(U)$ and $\beta_k > \alpha^1_k$. By Proposition \ref{cambioelemento}, there exists $\be^{\prime}(U) \in A_{D-1}(U)$ such that $\beta_k \leq \beta^{\prime}_k$ and $H_k(\beta^{\prime}(U)) \subseteq A(U)$. By what said above, for some $l\geq 1$, $H_k(\be^{\prime}(U))= H_k(\al^{l}(U)) $, hence we must have $$ \alpha_k^l = \beta^{\prime}_k \geq \beta_k > \alpha^1_k $$ and this is a contradiction since $\al^1(U) \geq \al^l(U)$.

Similarly, for $j > 1$, assume $\al^j \in A_i$ for some $i < D-j$. Find as above $\be, \be^{\prime} \in A_{D-j}$ and $l$ such that $$ \alpha_k^l = \beta^{\prime}_k \geq \beta_k > \alpha^j_k. $$ If $l \geq j$ this is a contradiction since in this case $\al^j(U) \geq \al^l(U)$. Otherwise, if $l < j$, by inductive hypothesis $\al^l(U) \subseteq A_{D-l}$, but since $\al^l(U)= \min H_k(\al^l(U)) $, we get $\be^{\prime}(U) \geq \al^l(U)$ and therefore $D - j \geq D - l$, again a contradiction.

Hence, we proved that all $U$-subspaces of $A$ are contained in the levels $A_i$ with $i \geq N - s_k + 1$ and this lower bound is sharp.
 
It remains to show that all the other subspaces of dimension $d-|U|$ are contained in the levels $A_i$ for $i\geq N-s_k+1$. This follows immediately observing that, if $T$ is a subset of $I$ of cardinality $|U|$, using the same proof seen above we can prove that $i\geq N-(\sum_{i\in T}w_i)+1 \geq N-s_k+1$.
We conclude the proof of the theorem observing that for $k=d$, the level $A_i$ contains a subspace of dimension zero if and only if $i\geq N-s_d+1$. But a subspace of dimension zero is a single element and hence, by definition of levels, is clear that $A_i$ contains a subspace of dimension zero if and only $i \geq 1$. This implies the thesis $N=s_d$. 
\end{proof}


\begin{ex}
\label{exap}
Consider the semigroup of Example \ref{ex3rami}, if we take $\bs{e}=(1,2,3)\in S$ and consider the ideal $E=S+\bs{e}$, then $A=\Ap(S,\bs{e})$. In this case $\bs{c_E}=\bs{c}+\bs{e}=(4,7,12)$ and $\bs{\gamma_E}=\bs{\gamma}+\bs{e}=(3,6,11)$.
\begin{eqnarray*}
A_1&=&\{(1,2,3)\}.\\
A_2&=&\{(1,2,6),(1,2,7),(2,3,3),(3,3,3)\}\cup\ (\infty,3,3).\\
A_3&=&\{(1,2,8),(2,3,6),(2,3,7),(2,4,3),(3,3,6),(3,3,7),(3,5,3),(3,6,3)\}\cup\\
&&\cup(\infty,3,6)\cup(\infty,3,7)\cup(\infty,5,3)\cup(\infty,6,3)\cup(3,\infty,3).\\
A_4&=&\{(3,5,11),(3,6,10)\}\cup(\infty,5,11)\cup(\infty,6,6)\cup(\infty,6,9)\cup(\infty,6,10)\cup\\ &&\cup(3,\infty,6)\cup(3,\infty,9)\cup(3,\infty,10)\cup(2,4,\infty)\cup(\infty,\infty,3).\\
A_5&=&\{(3,6,11)\}\cup(\infty,6,11)\cup(3,\infty,11)\cup(3,5,\infty)\cup(\infty,5,\infty)\cup(\infty,\infty,10).\\
A_6&=&(3,\infty,\infty)\cup(\infty,6,\infty)\cup(\infty,\infty,11).
\end{eqnarray*}
We notice that levels $A_6$, $A_5$, $A_4$ contain planes, level $A_3$, $A_2$ contain lines and only the level $A_1$ does not contain infinite subspaces.
\end{ex}

\section*{Acknowledgement}
The first author is supported by the NAWA Foundation grant Powroty "Applications of Lie algebras to Commutative Algebra".
The other two authors are funded by the project "Proprietà algebriche locali e globali di anelli associati a curve e ipersuperfici" PTR 2016-18 - Dipartimento di Matematica e Informatica - Università di Catania".
The authors wish to thank Marco D'Anna for the interesting and helpful discussions about the content of this article and to thank the referee for the careful revision and the comments that improved the quality of this article.

\end{document}